\documentclass[final,3p]{elsarticle}
\usepackage{eurosym}
\usepackage{amsfonts}
\usepackage{amsmath}
\usepackage{amssymb}
\usepackage{graphicx}
\usepackage[tableposition=top]{caption}
\usepackage{subcaption}
\usepackage{float}
\usepackage{setspace}
\usepackage{placeins}
\usepackage{mathrsfs}

\newsavebox \foobox
\newlength{\foodim}

\graphicspath{{C:/Figures/}}
\newtheorem{theorem}{Theorem}
\newtheorem{proof}{Proof}

\newtheorem{example}{Example}

\newtheorem{lemma}{Lemma}

\newtheorem{remark}{Remark}

\usepackage{multirow}

\journal{a Q1 journal}

\begin{document}
\begin{frontmatter}
\title{A high-order predictor-corrector method for initial value problems with fractional derivative involving Mittag-Leffler kernel: epidemic model case study}
\author{Sami Aljhani$^{a}$, Mohd Salmi Md Noorani$^{a}$, Redouane Douaifia$^{b}$, Salem Abdelmalek$^{b,c}$}
\address{$(a)$ Department of Mathematical Sciences, Faculty of Science and Technology, Universiti Kebangsaan Malaysia 43600 Bangi, Selangor, Malaysia \\
$(b)$ Laboratory of Mathematics, Informatics and Systems (LAMIS), Larbi Tebessi University, Tebessa, Algeria\\
	$(c)$ Department of Mathematics and Computer Science, Larbi Tebessi University, Tebessa, Algeria\\
	
	Emails: sami.aljhani@gmail.com, redouane.douaifia@univ-tebessa.dz, salem.abdelmalek@univ-tebessa.dz}

\begin{abstract}
In this paper, we propose a numerical scheme of the predictor-corrector type for solving nonlinear fractional initial value problems, the chosen fractional derivative is called the Atangana-Baleanu derivative defined in Caputo sense (ABC). This proposed method is based on Lagrangian quadratic polynomials to approximate the nonlinearity implied in the Volterra integral which is obtained by reducing the given fractional differential equation via the properties of the ABC-fractional derivative. Through this technique, we get corrector formula with high accuracy which is implicit as well as predictor formula which is explicit and has the same precision order as the corrective formula. On the other hand, the so-called memory term is computed only once for both prediction and correction phases, which indicates the low cost of the proposed method. Also, the error bound of the proposed numerical scheme is offered.\\
Furthermore, numerical experiments are presented in order to assess the accuracy of the new method on two differential equations. Moreover, a case study is considered where the proposed predictor-corrector scheme is used to obtained approximate solutions of ABC-fractional generalized SI (susceptible-infectious) epidemic model for the purpose of analyzing dynamics of the suggested system as well as demonstrating the effectiveness of the new method to solve systems dealing with real-world problems.
\end{abstract}
\end{frontmatter}

\section{Introduction}
Despite the fact that fractional calculus has a lengthy history in mathematics,  it has only recently seen a considerable number of real-world applications.  Fractional derivatives are important due to their attractive characteristics (e.g. memory effect), including their wide dynamical range \cite{alquran2018novel,douaifia2020asymptotic,aljhani2020numerical}.
Many definitions of fractional derivatives and integrals exist,  including Riemann-Liouville, Caputo, Grunwald-Letnikov \cite{magin2006fractional}. In 2015, Caputo and Fabrizio (CF) have suggested a new idea of fractional derivatives based on the exponential decay \cite{caputo2015new}.  Following that,  Atangana and Baleanu (AB) proposed a novel concept of fractional derivatives with non-singular and non-local kernel based on the Mittag-Leffler function \cite{atangana2016new}. 
Due to the difficulty or impossibility of obtaining to obtain exact solution of fractional differential equations, numerical techniques are required to provide approximate solutions.  Over the past few decades,  numerical methods to solve initial value problems have attracted great interest from the research community,
which support various disciplines,  including medicine, chemistry, biology, physics,  business,  and the arts, etc. Where computational algorithms are implemented. Moreover, numerical analysis techniques are the perfect tools to evaluate the behavior of real-life complicated models. Therefore, the implementation of accurate numerical methods is suibtle in the problems representing real world phenomena \cite{atangana2016new,aljhani2021numerical,Toufik2017,ullah2020modeling}.  Over 2000 years ago, the simplest method in the literature was introduced with the concept of interpolation. In contrast, several other techniques were introduced for nonlinear systems,  such as Newton’s method, Euler’s method, Gaussian elimination,  and Lagrange interpolation \cite{werner1984polynomial,dimitrov1994note,yang2016visualizing,krogh1970efficient}. Over the past few years, several numerical tools have been developed for solving fractional ordinary differential equations with general nonlinearity  \cite{diethelm1998fracpece,diethelm2004detailed,djida2017numerical,baleanu2018nonlinear,Nguyen2017,Lee2021,li2011numerical}. Such as the fractional Adams-Bashforth method, which is the most popular scheme in which the basis of Lagrange interpolation is used  \cite{jain2018numerical,zhang2018decoupled} and the fractional Adams-Bashforth/Moulton method developed with Newton linear/quadratic interpolations  \cite{atangana2021new, douaifia2021newton}. 
Baleanu et al.  \cite{baleanu2018nonlinear} used the definition of the Atangana-Baleanu derivative in the Caputo sense to produce a predictor-corrector approach for fractional differential equation as well as to prove the existence and uniqueness of such problem. They have converted the ABC-fractional ordinary differential equation into a Volterra Integral equation,  then they have used the rectangle rule for the predictor and the trapezoid rule for the corrector. 
On the other hand, based on the techniques and approximations appeared in \cite{Nguyen2017,Lee2021} for finding approximate solutions of fractional differential equations with Caputo and Caputo-Fabrizio fractional derivatives, respectively. We can to propose a new and efficient predictor-corrector method for solving ABC-fractional initial value problems by using quadratic lagrange interpolation, memory term, and second-order Taylor’s expansion.

The following is a breakdown of the paper's structure. Basic definitions and notations, including the Atangana-Baleanu fractional derivatives and integral, are introduced in section 2.  In section 3,  a new numerical method is proposed for solving fractional initial value problems involving the Atangana-Baleanu fractional derivative in Caputo sense. The schemes' error estimates are obtained in section 4. Finally, in section 5, we present numerical examples to demonstrate the efficacy of the proposed scheme,  and we also compare the solutions with other methods \cite{Toufik2017,baleanu2018nonlinear}. Also, the stability analysis of the ABC-fractional-order SI (Susceptible-Infection) epidemiological model and its comparison with the numerical results obtained via the proposed scheme, which investigates the significance of disease prevalence in the community, is discussed in the last section.

\section{Preliminaries}
The Atangana-Baleanu fractional derivative in the Caputo sense (ABC) is
defined as (cf. \cite{atangana2016new}):
\begin{equation}\label{fractional_D_AB}
{}^{ABC}_{\quad \,0}D^{\alpha}_{t}f(t)=\displaystyle \frac{AB(\alpha)}{1-\alpha}\int^{t}_{0}f^\prime(s)
\textbf{E}_{\alpha}\left(\frac{-\alpha}{1-\alpha}(t-s)^{\alpha}\right) ds,
\end{equation}
where $\alpha\in (0,1)$, $AB(\alpha)>0$ is a normalization function obeying
$AB(0)=AB(1)=1$ \\ 
$\left(\text{e.g. } AB(\alpha)=1-\alpha+\displaystyle \frac{\alpha}{\Gamma(\alpha)}\right) $, $ \textbf{E}_{\alpha}(.)$ denotes the Mittag-Leffler function of order $\alpha $ defined by
\begin{equation}
\begin{matrix}
\textbf{E}_{\alpha}(z)=\displaystyle\sum^{\infty}_{k=0}\frac{{z}^{k}}{\Gamma (\alpha k+1)}, & z,\alpha\in \mathbb{C},& \text{and }Re(\alpha)>0,
\end{matrix}
\end{equation}
and $\Gamma(.)$ denotes Gamma function, defined as
\begin{equation}
\begin{matrix}
\Gamma(\alpha)=\displaystyle\int^{+\infty}_{0}t^{\alpha-1} e^{-t} dt, &  \qquad Re(\alpha)>0.
\end{matrix}
\end{equation}
The fractional integral for the ABC, which is newly defined with a nonlocal kernel and does not have singularities at $t=s$, is
defined as follows (cf. \cite{atangana2016chaos}):

\begin{align}\label{fractional_I_AB}
&{}^{ABC}_{\quad \, 0}I^{\alpha}_{t}f(t)=\frac{1-\alpha}{AB(\alpha)}f(t)+\frac{\alpha}{AB(\alpha)\Gamma(\alpha)}\int^{t}_{0}f(s)(t-s)^{\alpha-1}ds.
\end{align}

%Here,  when $\alpha$ tends to zero,  the initial function is recovered, and when $\alpha$ tends to unity, the classical ordinary integral obtained.

We consider the initial-value problem with Atangana-Baleanu derivative 
\begin{equation}\label{Main_Systm}
\begin{cases}
&{}^{ABC}_{\quad \, 0}D^{\alpha}_{t} y\left(t\right)= f\left(t,y\left(t\right)\right),   \qquad   \qquad    \qquad       0< t< T < \infty, \\ \\
& \qquad  \quad \, y\left(0\right)= y_{0},   
\end{cases}
\end{equation}
where $f$ is a smooth nonlinear function that guarantees the existence of a unique solution for (\ref{Main_Systm}), with the fractional order $\alpha \in \left(0,1\right)$, and $y_{0}\in\mathbb{R}$. A continuous function $y(t)$ is the solution of (\ref{Main_Systm})  if and only if it is the solution of following Voltera-integral equation:
\begin{equation}\label{Sol_Main_Systm_Voltera_int}
\begin{matrix}
y(t) & =\displaystyle y_{0} + \frac{1 - \alpha}{AB(\alpha)} f (t,y(t)) + \frac{\alpha}{AB(\alpha)\Gamma(\alpha)} \int_{0}^{t} f\left(s,y(s)\right) \left(t - s\right)^{\alpha - 1} ds.
\end{matrix}
\end{equation}
At point $t_{n + 1}=h(n+1)$, $n=0,1,\dots,N\in \mathbb{N}$ with $h=\frac{T}{N}$, we get
\begin{equation}\label{Sol_MainS_Voltera_int_At_nplus1}
\begin{split}
y_{n + 1}  & = y_{0} + \frac{1 - \alpha}{AB(\alpha)} f_{n + 1} + \frac{\alpha}{AB(\alpha)\Gamma(\alpha)} \int_{0}^{t_{n+1}} f\left(s,y(s)\right) \left(t_{n + 1} - s\right)^{\alpha - 1} ds \\
& = y_{0} + \frac{1 - \alpha}{AB(\alpha)} f_{n + 1} + \frac{\alpha}{AB(\alpha)\Gamma(\alpha)} \left( \int_{0}^{t_{n}}f\left(s,y(s)\right) \left(t_{n + 1} - s\right)^{\alpha - 1} ds + \int_{t_{n}}^{t_{n+1}} f\left(s,y(s)\right) \left(t_{n + 1} - s\right)^{\alpha - 1} ds\right) 	\\	
&	= y_{0} + \frac{1 - \alpha}{AB(\alpha)}f_{n + 1} + y^{lag}_{n + 1}	+ y^{inc}_{n + 1},
\end{split}
\end{equation}
where, $y^{lag}_{n + 1}$ and $y^{inc}_{n + 1}$ are called lag term and increment term respectively, which take the following forms:
\begin{equation}\label{y_lag}
y^{lag}_{n + 1} :=\displaystyle \frac{\alpha}{AB(\alpha)\Gamma(\alpha)} \int_{0}^{t_{n}} (t_{n + 1} - s)^{\alpha - 1}f(s,y(s)) ds,
\end{equation}
and
\begin{equation}\label{y_inc}
y^{inc}_{n + 1}  :=\displaystyle \frac{\alpha}{AB(\alpha)\Gamma(\alpha)} \int_{t_n}^{t_{n+1}} (t_{n + 1} - s)^{\alpha - 1}f(s,y(s)) ds.
\end{equation}

\section{Predictor-corrector scheme with quadratic interpolation}
To start presenting the proposed numerical method to solve the initial-value problem (\ref{Main_Systm}), we need the following lemma: 
\begin{lemma} (\cite{Nguyen2017}) \label{Lemma_incrTerm}
	Assume that $\psi$ $\in$ $\mathbb{P}_{2}([0,T])$, where $\mathbb{P}_{2}([0,T]) $ is the space of all polynomials of degree less than or equal to two. Let $\psi_{k}$, $k = 0,....,N$ be the restricted value of $\psi(t)$ on $t_{k}$ ( $0 \leq k \leq N$ ). Then there exist reals $b^{0}_{n+1}, b^{1}_{n+1}$ and $b^{2}_{n+1}$, such that 
	\begin{equation}\label{Exactvalue_incrTermInt}
	\int_{t_{n}}^{t_{n+1}} (t_{n+1} - s)^{\alpha - 1} \psi (s) ds  = B \sum_{j=0}^{2} b^{j}_{n+1} \psi_{n+j-2}.
	\end{equation}
	Where, $B=\displaystyle\frac{h^{\alpha}}{\alpha(\alpha + 1)(\alpha + 2)}, b^{0}_{n +1} = \frac{\alpha + 4}{2}, b^{1}_{n + 1} = -2(\alpha + 3), b^{2}_{n + 1} = \frac{2\alpha^{2} + 9\alpha + 12}{2}$.
\end{lemma} 

Now, we use quadratic Lagrange polynomial of $f(s,y(s))$ over the intervals $[t_{i-1} , t_{i+1}]$ , $1 \leq i \leq N-1$:
\begin{align}\label{Quad_interp_f}
f(s,y(s)) \approx \sum_{l=i-1}^{i+1}f_{l}q^{i}_{l}(s),
\end{align}
where
\begin{align}\label{Quad_interp_q}
q^{i}_{l}(s) = \prod_{\substack{k=i-1 \\ l\neq k}}^{i + 1} \frac{s - t_{k}}{t_{l} - t_{k}}.
\end{align}
But, on $[t_{0} , t_{1}]$ we can interpolate $f(s,y(s))$ with the points $\big\{{t_{0}, t_{\frac{1}{2}}, t_{1}}\big\}$, then we get 
\begin{align}\label{Quad_interp_f_tZeroOne}
f(s,y(s)) \approx f_{0} q^{0}_{0} (s) + f_{\frac{1}{2}} q^{0}_{\frac{1}{2}}(s) + f_{1}q^{0}_{1}(s),
\end{align}
where
\begin{align}\label{Quad_interp_q_tZeroOne}
q^{0}_{0}(s) =\displaystyle \frac{(s - t_{0})(s - t_{\frac{1}{2}})}{(t_{1} - t_{0})(t_{1} - t_{\frac{1}{2}})} , \ q^{0}_{\frac{1}{2}}(s) = \frac{(s - t_{0})(s - t_{1})}{(t_{\frac{1}{2}} - t_{1})(t_{\frac{1}{2}} - t_{1})} , \	q^{0}_{1}(s) = \frac{(s - t_{0})(s - t_{\frac{1}{2}})}{(t_{1} - t_{0})(t_{1} - t_{\frac{1}{2}})} .
\end{align}
The approximation of $y(t_{n+1})$ denoting by $\widetilde{y}_{n+1}$, according to (\ref{Quad_interp_f})-(\ref{Quad_interp_q_tZeroOne}), $\widetilde{y}_{n+1}$ takes the following form:
\begin{equation}\label{y_corrector}eq216
\widetilde{y}_{n+1} =\displaystyle y_{0} + {}^{AB} \widetilde{f}^{P}_{n+1} + \widetilde{y}^{lag}_{n+1} + \widetilde{y}^{inc}_{n+1},
\end{equation}
where
\begin{equation}\label{eq211}
{}^{AB} \widetilde{f}^{P}_{n+1}  =\displaystyle \frac{1 - \alpha}{AB(\alpha)} f(t_{n+1} , \widetilde{y}^{P}_{n+1}),
\end{equation}

\begin{equation}\label{eq212}
\widetilde{y}^{lag}_{n+1} =\displaystyle \frac{\alpha}{AB(\alpha)\Gamma(\alpha)} \left(\sum_{j=0}^{2} a^{j,0}_{n+1}          \widetilde{f}_\frac{j}{2} + \sum_{j=0}^{2}\sum_{k=1}^{n-1} a^{j,k}_{n+1}  \widetilde{f}_{k+j-1}\right)
\end{equation}
and
\begin{equation}\label{eq213}
\widetilde{y}^{inc}_{n+1}  =\displaystyle \frac{\alpha}{AB(\alpha)\Gamma(\alpha)} \left(\sum_{j=0}^{1} a^{j,n}_{n+1}   \widetilde{f}_{n+j-1} + a^{2,n}_{n+1}     \widetilde{f}^{P}_{n+1}\right)
\end{equation}
with
\begin{equation}\label{eq214}
\begin{matrix}
a^{j,k}_{n+1} =\displaystyle \int_{t_{k}}^{t_{k+1}}(t_{n+1} - s)^{\alpha-1} q^{k}_{\widehat{j}}(s) ds, &\textit{for}&  j \in \left\{0,1,2\right\},
\end{matrix}
\end{equation}
where,
\begin{equation}\label{eq215}
\widehat{j} = \begin{cases} 
\displaystyle \frac{j}{2} &\text{if } k=0, \\  
j+k-1 & \text{if } 1 \leq k \leq n.
\end{cases} 
\end{equation}

The predictor term can be approximated as follows: 
\begin{equation}\label{eq216}
\widetilde{y}^{P}_{n+1} = y_{0} + {}^{AB} \widetilde{f}^{apprx}_{n+1} + \widetilde{y}^{lag}_{n+1} + \frac{\alpha h^{\alpha}}{AB(\alpha)\Gamma(\alpha+3)}  \sum_{j=0}^{2} b^{j}_{n+1} \widetilde{f}_{n+j-2}  ,
\end{equation}
where 
\begin{equation}\label{eq217}
{}^{AB} \widetilde{f}^{apprx}_{n+1} = \frac{1-\alpha}{AB(\alpha)} \left(f_{n-2} -3 f_{n-1} -3 f_{n}\right) , \hspace{15pt}  n \geq 2.
\end{equation}

\section{Error analysis}
Throughout this section, we need the following lemmas: 
\begin{lemma} \label{lemma22}
	Let $f$ $\in$ $C^{n+1}  ([a,b]) $ and $P_{n}$ $\in$ $\mathbb{P}_{n}([a,b])$, (where $[a,b]\subset \mathbb{R}$ and $\mathbb{P}_{n}([a,b]) $ is the space of all polynomials of degree less than or equal to $n$)  interpolate the function $f$ at $t_{k}$, $ 0 \leq k \leq N$, with $t_{0} = a, t_{N} = b, $, then there exists $\xi$ $\in  (a,b) $ such that, for any $s$ $\in [a,b]$ 
	\begin{align}\label{eq218}
	f(s) - P_{n}(s) = \frac{f^{n+1}(\xi)}{(n+1)!} \prod_{k=0}^{n}(s-t_{k}).
	\end{align}
\end{lemma} 
\begin{lemma} (\cite{Nguyen2017}) \label{lemma23}
	For $\delta>0$, we have
	\begin{equation}\label{eq219}
	\sum_{k=0}^{n-1}\int_{t_{k}}^{t_{k+1}}(t_{n+1} - s)^{\delta - 1} ds \leq \frac{T^{\delta}}{\delta}.
	\end{equation}
\end{lemma} 
\begin{lemma} (\cite{Dixon1986}) \label{lemma24} (Discrete Gronwall's Inequality) 
	Let $\left\{a_{n}\right\}^{N}_{n=0}  , \left\{b_{n}\right\}^{N}_{n=0}$ be non-negative sequences with second one is monotonic increasing and satisfy that 
	\begin{equation}\label{eq220}
	\begin{matrix}
	a_{n} \leq \displaystyle b_{n} + M h ^{\theta} \sum_{k=0}^{n-1}(n-k)^{\theta-1}a_{k}, & 0 \leq n \leq N ,
	\end{matrix}
	\end{equation}
	where, $M>0$ is independent of $h>0$, and $0<\theta\leq 1$. Then,
	\begin{equation}\label{eq221}
	a_{n} \leq b_{n} \textbf{E}_{\theta} \big( M \Gamma (\theta)(nh)^{\theta}\big).
	\end{equation}
\end{lemma} 

\begin{lemma} (\cite{Nguyen2017}) \label{lemma25}
	There exist $K_{1} , K_{2} >0$ such that for $\alpha \in (0,1)$, and $j=0,1,2$, we have
	\begin{equation}\label{eq222}
	\Bigl| a^{j,k}_{n+1}\Bigl|   \leq \begin{cases} 
	\displaystyle K_{1}(n-k)^{\alpha-1} h^{\alpha} &\text{if  } 0 \leq k \leq n-1, \\  
	K_{2} & \text{if  }  k = n.
	\end{cases} 
	\end{equation}
\end{lemma} 
Let $T^{P}_{n+1}$ be the truncation error of prediction at point $t_{n+1}$, defined by
\begin{align}\label{eq224}
T^{P}_{n+1} = & \Biggl| \frac{\alpha}{AB(\alpha)\Gamma(\alpha)} \int_{0}^{t_{n+1}} (t_{n+1} - s)^{\alpha - 1} f (s,y(s))ds - \frac{\alpha}{AB(\alpha)\Gamma(\alpha)} \sum_{j=0}^{2} a^{j,0}_{n+1} f_\frac{j}{2} \\  \nonumber & -\frac{\alpha}{AB(\alpha)\Gamma(\alpha)} \sum_{j=0}^{2} \sum_{k=1}^{n-1} a^{j,k}_{n+1} f_{k+j-1} - \frac{\alpha h^{\alpha}}{AB(\alpha)\Gamma (\alpha + 3)} \sum_{j=0}^{2} b^{j}_{n+1} f_{n+j-2} \Biggl|.
\end{align}
\begin{theorem} \label{theorem21}
	Assume that $f(.,y(.)) \in  C^{3} ([0,T]). $ Then, there exists $C>0$ (independent of all grid parameters) such that:
	\begin{align}\label{eq225}
	T^{P}_{n+1} \leq C h^{3}.
	\end{align}
\end{theorem} 
\begin{proof}
	We set the notation, $AB_{g}=\displaystyle \frac{\alpha}{AB(\alpha)\Gamma(\alpha)}$, then 
	\begin{equation}\label{eq226}
	T^{P}_{n+1} \leq \displaystyle \sum_{j=1}^{3} \mathcal{I}_{j},
	\end{equation}
	where,
	\begin{align}
	\mathcal{I}_{1} &:= AB_{g} \int_{0}^{t_{1}} (t_{n+1} - s)^{\alpha-1} \Bigl|  f(s,y(s)) - \sum_{j=0}^{2} f_{\frac{j}{2}} q^{0}_{\frac{j}{2}}(s) \Bigl|  ds, \\  
	\mathcal{I}_{2} &:= AB_{g} \sum_{k=1}^{n-1} \int_{t_{k}}^{t_{k+1}} (t_{n+1} - s)^{\alpha-1} \Bigl|  f(s,y(s)) - \sum_{j=0}^{2} f_{k+j-1} q^{k}_{k+j-1} (s) \Bigl|  ds, \\ 
	\mathcal{I}_{3} &:= \Bigl|  AB_{g} \int_{t_{n}}^{t_{n+1}}(t_{n+1} -s)^{\alpha-1} f(s,y(s)) - \Gamma(\alpha)AB_{g} \frac{h^{\alpha}}{\Gamma(\alpha + 3)} \sum_{j=0}^{2} b^{j}_{n+1}f_{n+j-2} \Bigl|.\label{mathcal_I_three}
	\end{align}
	Thanks to lemma \ref{lemma22}, there exists $C_{1}>0$, such that
	\begin{equation}\label{est_I_1_2}
	\begin{matrix}
	\mathcal{I}_{j} \leq C_{1}h^{3}, & j=1,2.
	\end{matrix}
	\end{equation}
	The Taylor expansion of $f$ around $t_{n}$ gives:
	\begin{equation}\label{eq227}
	\begin{matrix}
	f(t) = \displaystyle P_{2}(t) + \frac{1}{3!}f^{\prime\prime\prime} (\xi_{n})(t-t_{n})^{3} + O(h^{3}), & \xi_{n} \in (t_{n} , t),
	\end{matrix}
	\end{equation}
	where,
	\begin{equation}
	\begin{matrix}
	P_{2}(t) =\displaystyle f_{n} + f^{\prime}(t_{n})(t-t_{n}) + \frac{1}{2} f^{\prime\prime} (t_{n}) (t-t_{n})^{2}, & (P_{2} \in \mathbb{P}_{2}([0,T])).
	\end{matrix}
	\end{equation}
	According to lemma \ref{lemma22} and (\ref{mathcal_I_three}), we have
	\begin{align}\label{eq228}
	\mathcal{I}_{3} =&  \Bigl|  AB_{g} \int_{t_{n}}^{t_{n+1}} (t_{n+1} - s)^{\alpha-1}(P_{2}(s) + \frac{1}{3!}f^{\prime} (\xi_{n})(s - t_{n})^{3} + O(h^{3})) ds \notag\\  
	&- \Gamma(\alpha)AB_{g} \frac{h^{\alpha}}{\Gamma(\alpha + 3)} \sum_{j=0}^{2} b^{j}_{n+1}f_{n+j-2}  \Bigl| ,
	\end{align}
	it follows that
	\begin{align}\label{eq229}
	\mathcal{I}_{3}  &\leq  \Gamma (\alpha) AB_{g} \frac{h^{\alpha}}{\Gamma(\alpha+3)}\Big(\mid b^{0}_{n+1} \mid \mid P_{2}(t_{n-2}) - f_{n-2} -f_{n-2}\mid + \mid b^{1}_{n+1} \mid \mid P_{2}(t_{n-1}) - f_{n-1} \mid\Big) \notag \\  
	& \ \ \ +  AB_{g} \left|  \int_{t_{n}}^{t_{n+1}}(t_{n+1}-s)^{\alpha -1}\Big(\frac{1}{6}f^{\prime\prime\prime}(\xi_{n}){(s-t_{n})^{3}}\Big)ds\right|  + O(h^{3}) .
	\end{align}
	According to  (\ref{eq226}) and  (\ref{eq227}), we have
	\begin{equation}\label{eq230}
	\begin{matrix}
	\Bigl|  P_{2}(t_{n-2}) - f_{n-2}  \Bigl|  \leq \displaystyle \frac{4 M h^{3}}{3}, & \text{where } M :=\max\left\lbrace  \mid f^{\prime\prime\prime}(\xi_{k}) \mid \text{ : } 0 \leq k \leq N\right\rbrace,
	\end{matrix}
	\end{equation} 
	and
	\begin{equation}\label{eq231}
	\Bigl| P_{2}(t_{n-1}) - f_{n-1} 	 \Bigl|   \leq \frac{M}{6}h^{3}.
	\end{equation}
	We thus obtain the estimate
	\begin{equation}\label{eq232}
	\mathcal{I}_{3} \leq \Gamma (\alpha) AB_{g} \frac{h^{\alpha}}{\Gamma(\alpha + 3)} \Big(  b^{0}_{n+1}  \frac{4 M h^{3}}{3} + \left|  b^{1}_{n+1}\right| \frac{M}{6}h^{3}\Big)  +  \frac{M h^{3+\alpha}}{6AB(\alpha)\Gamma(\alpha)} + O(h^{3}).
	\end{equation}
	Consequently, there exists $C_{2}>0$ such that
	\begin{equation}\label{eq233}
	\mathcal{I}_{3}\leq C_{2}h^{3}.
	\end{equation}
	According to (\ref{eq226}), (\ref{est_I_1_2}) and (\ref{eq233}), proof of theorem is achieved.
\end{proof}

\begin{theorem}\label{theorem22} (Global Predictor Error) 
	Assume that $f(.,y(.)) \in C^{3} ([0,T])$ and is Lipschitz continuous in its second argument, i.e.
	\begin{equation}\label{eq234}
	\begin{matrix}
	\exists L > 0, \textit{ such that }\left|  f(t,y_{1}) - f(t,y_{2}) \right|  \leq L \left|  y_{1} - y_{2} \right|,  & \forall y_{1} , y_{2} \in \mathbb{R}.
	\end{matrix}
	\end{equation}
	Then, there exist $K_{1},K_{2}>0$ such that, the global predictor error satisfies
	\begin{equation}\label{eq235}
	E^{P}_{n+1} := \left|  y_{n+1} - \widetilde{y}^{P}_{n+1} \right| \leq T^{P}_{n+1} + K_{1}  AB_{g}L h^{\alpha} E_{\frac{1}{2}} +  K_{2}  AB_{g}Lh^{\alpha} \sum_{k=1}^{n}(n-k+1)^{\alpha-1}E_{k} + O(h^{3}).
	\end{equation}
	Where, $AB_{g}:=\displaystyle \frac{\alpha}{AB(\alpha)\Gamma(\alpha)}$.
\end{theorem} 
\begin{proof}
	We set $^{AB}{f}_{n+1} := \displaystyle\frac{1-\alpha}{AB(\alpha)} f(t_{n+1},y_{n+1})$. Thus 
	\begin{equation}\label{eq236}
	E^{P}_{n+1} =  \left| {}^{AB}{f}_{n+1} + AB_{g} \int_{0}^{n+1}(t_{n+1} - s)^{\alpha - 1}f(s,y(s))ds - {}^{AB} \widetilde{f}^{apprx}_{n+1} - \widetilde{y}^{lag}_{n+1} -\displaystyle\frac{\Gamma(\alpha)AB_{g}h^{\alpha}}{\Gamma(\alpha+3)}\sum_{j=0}^{2}b^{j}_{n+1}\widetilde{f}_{n+j-2}  \right|,
	\end{equation}
	therefore
	\begin{align}\label{eq237}
	E^{P}_{n+1}  =  \Biggl| & {} ^{AB}{f}_{n+1} + AB_{g} \int_{0}^{t_{n+1}} (t_{n+1}- s)^{\alpha-1} f(s,y(s))ds - AB_{g} \sum_{j=0}^{2} a^{j,0}_{n+1} f_{\frac{j}{2}} \nonumber \\ \nonumber & -AB_{g}\sum_{j=0}^{2}\sum_{k=1}^{n - 1} a^{j,k}_{n+1}f_{k+j-1} - AB_{g}h^{\alpha} \sum_{j=0}^{2}b^{j}_{n+1} f_{n+j-2} + AB_{g} \sum_{j=0}^{2} a^{j,0}_{n+1}f_{\frac{j}{2}} \\ \nonumber & + AB_{g}\sum_{j=0}^{2}\sum_{k=1}^{n - 1} a^{j,k}_{n+1}f_{k+j-1} - \displaystyle\frac{\Gamma(\alpha)AB_{g}h^{\alpha}}{\Gamma(\alpha+3)} \sum_{j=0}^{2}b^{j}_{n+1} f_{n+j-2} -  {}^{AB} \widetilde{f}^{apprx}_{n+1} - \widetilde{y}^{lag}_{n+1} \\  & - \displaystyle\frac{\Gamma(\alpha)AB_{g}h^{\alpha}}{\Gamma(\alpha+3)} \sum_{j=0}^{2} b^{j}_{n+1} \widetilde{f}_{n+j-2}  \Biggl| ,
	\end{align}
	it follows that
	\begin{align}\label{eq238}
	E^{P}_{n+1} & \leq T^{P}_{n+1} + \left|  {}^{AB}{f}_{n+1} - {}^{AB}{\widetilde{f}}_{n+1}^{apprx} \right|  \\ \nonumber
	& \ \ \ + AB_{g} \sum_{j=0}^{2} \left|  a^{j,0}_{n+1}  \right|  \left|  f_{\frac{j}{2}} - \widetilde{f}_{\frac{j}{2}} \right|   + AB_{g} \sum_{j=0}^{2}\sum_{k=1}^{n-1} \left|  a^{j,k}_{n+1} \right|  \left|  f_{k+j-1} - \widetilde{f}_{k+j-1} \right|    \\ \nonumber 
	& \ \ \ + \frac{\Gamma(\alpha)AB_{g}h^{\alpha}}{\Gamma(\alpha + 3)}  \sum_{j=0}^{2} \left|  b^{j}_{n+1} \right|  \left|  f_{n+j-2} - \widetilde{f}_{n+j-2} \right| .
	\end{align}
	Hence, according to the lemma \ref{lemma23} there exist $C_{1},K_{1}>0$, such that
	\begin{align}\label{eq239}
	E^{P}_{n+1} & \leq  T^{P}_{n+1} + C_{1} \sum_{j=0}^{2} E_{n+j-2} + K_{1}  AB_{g}Lh^{\alpha} n^{\alpha-1}(E_{\frac{1}{2}} + E_{1})  \notag\\  
	& \ \ \ +  K_{1}AB_{g}Lh^{\alpha} \sum_{j=0}^{2}\sum_{k=1}^{n-1}(n-k)^{\alpha - 1} E_{k+j-1} + \frac{\Gamma(\alpha)AB_{g}Lh^{\alpha}}{\Gamma(\alpha+3)}\sum_{j=0}^{2} \left|  b^{j}_{n+1} \right|   E_{n+j-2} + O(h^{3}).
	\end{align}
	On the other hand, for $\alpha\in (0,1)$ we have
	\begin{equation}
	\begin{matrix}
	(n+1-i)^{\alpha-1} + (n-1)^{\alpha-1} + (n-i-1)^{\alpha-1} \leq 6(n+1-i)^{\alpha-1},  & 1 \leq i \leq n-2.
	\end{matrix}
	\end{equation}
	Consequently, there exists $K_{2}>0$ such that the estimate (\ref{eq239}), becomes
	
	\begin{equation}\label{eq240}
	E^{P}_{n+1} \leq T^{P}_{n+1} + K_{1}  AB_{g}L h^{\alpha} E_{\frac{1}{2}} +  K_{2}  AB_{g}Lh^{\alpha} \sum_{k=1}^{n}(n+1-k)^{\alpha-1}E_{k} + O(h^{3}).
	\end{equation}
\end{proof}

\begin{theorem} \label{theorem23}Assume that $f(.,y(.)) \in C^{3}([0,T]).$ Then there exists positive constants $C_{1}$ and $C_{2}$ (independent of grid parameters) such that: 
	\begin{align}\label{eq241}
	T^{C}_{n+1} \leq C_{1}h^{3} + C_{2}E^{P}_{n+1},
	\end{align}
	where, 
	\begin{align}\label{eq242}
	T^{C}_{n+1} & = \Biggl|  AB_{g} \int_{0}^{t_{n+1}}(t_{n+1}-s)^{\alpha-1}  f(s,y(s)) ds - AB_{g} \sum_{j=0}^{2} a^{j,0}_{n+1}f_{\frac{j}{2}} \notag  \\  
	& \ \ \ \ -AB_{g} \left( \sum_{j=0}^{2}\sum_{k=1}^{n-1}a^{j,k}_{n+1} f_{k+j-1} + \sum_{j=0}^{1} a^{j,n}_{n+1} f_{n+j-1} + a^{2,n}_{n+1}f^{P}_{n+1} \right)  \Biggl|,
	\end{align}
	with, $AB_{g} =\displaystyle
	\frac{\alpha}{AB(\alpha)\Gamma(\alpha)}$.
\end{theorem} 
\begin{proof}
	It follows immediately that,
	\begin{align}\label{eq243}
	T^{C}_{n+1} & \leq AB_{g} \int_{t_{0}}^{t_{1}}(t_{n+1} - s)^{\alpha-1} \Bigl|  f(s,y(s)) - \sum_{j=0}^{2} q^{0}_{\frac{j}{2}}(s) f_{\frac{j}{2}}\Bigl|  ds \notag \\ 
	& \ \ \ + AB_{g} \sum_{k=1}^{n-1} \int_{t_{k}}^{t_{k+1}}(t_{n+1} - s)^{\alpha-1} \Bigl|  f(s,y(s)) - \sum_{j=0}^{2} q^{k}_{k+j-1}(s) f_{k+j-1} \Bigl|  ds \notag \\ 
	& \ \ \ + AB_{g}  \int_{t_{n}}^{t_{n+1}}(t_{n+1} - s)^{\alpha-1} \Bigl|  f(s,y(s)) - \sum_{j=0}^{1} q^{n}_{n+j-1} (s) f_{n+j-1} - q^{n}_{n+1}(s) f^{P}_{n+1} \Bigl|  ds.
	\end{align}
	Thus, according to the lemma \ref{lemma22}, there exist $\widehat{C},\widetilde{C}>0$ such that
	\begin{equation}\label{eq244}
	T^{C}_{n+1} \leq \frac{\widetilde{C}T^{\alpha}}{6AB(\alpha)\Gamma(\alpha)} h^{3} + \frac{\widehat{C}h^{\alpha}}{AB(\alpha)\Gamma(\alpha)} E^{P}_{n+1} + O(h^{3}).
	\end{equation}
	The desired estimate (\ref{eq242}) can be derived by direct consideration of the last inequality (\ref{eq244}). 
\end{proof}

\begin{theorem}\label{theorem24}  (Global Error of the proposed method) With the same  assumptions as those of Theorem \ref{theorem22}. Then, we have 
	\begin{equation}\label{eq245}
	E_{n+1} := \left|  y_{n+1} - \widetilde{y}_{n+1} \right|  \leq Ch^{3},
	\end{equation}
	where $C>0$ (independent of grid parameters), given $E_{1}, E_{2} \leq \widehat{C}_{1}h^{3}$, and   $E_{\frac{1}{2}} \leq \widehat{C}_{2}h^{3-\alpha}$, with $\widehat{C}_{1},\widehat{C}_{2}>0$.
\end{theorem}
\begin{proof} 
	By taking into account the previous results, there exist $C_{j}>0$ (with $j=1,\dots,10$), such that 
	\begin{align}\label{eq246}
	E_{n+1} & \leq T^{C}_{n+1} +AB_{g} \left|  f^{P}_{n+1} - \widetilde{f}^{P}_{n+1} \right|  + \frac{L(1-\alpha)}{AB(\alpha)} E^{P}_{n+1} + AB_{g} \sum_{j=0}^{2} \left|  a^{j,0}_{n+1} \right|  \left|  f_{\frac{j}{2}} - \widetilde{f}_{\frac{j}{2}} \right|  \notag \\ 
	& \ \ \ +AB_{g}\sum_{j=0}^{2} \sum_{k=1}^{n-1} \left|  a^{j,k}_{n+1} \right|  \left|  f_{k+j-1} - \widetilde{f}_{k+j-1} \right|  + AB_{g} \sum_{j=1}^{1} \left|  a^{j,n}_{n+1} \right|  \left|  f_{n+j-1} - \widetilde{f}_{n+j-1} \right| \notag  \\  
	&\leq T^{C}_{n+1} + \frac{L(1-\alpha)}{AB(\alpha)} E^{P}_{n+1} + AB_{g} Lh^{\alpha} E_{\frac{1}{2}} + \frac{\widetilde{C}_{1}h^{\alpha}}{AB(\alpha)\Gamma(\alpha)} E^{P}_{n+1} \notag \\  
	& \ \ \ + AB_{g} \widetilde{C}_{2} Lh^{\alpha}  \sum_{k=1}^{n}(n-k+1)^{\alpha-1} E_{k} + O(h^{3}) \notag \\  
	&  \leq \widetilde{C}_{3} h^{3} + \widetilde{C}_{4}E^{P}_{n+1} + \frac{L(1-\alpha)}{AB(\alpha)} E^{P}_{n+1} + AB_{g} Lh^{\alpha}E_{\frac{1}{2}} \notag\\  
	& \ \ \ + \frac{\widetilde{C}_{5}h^{\alpha}}{AB(\alpha)}E^{p}_{n+1} + AB_{g} \widetilde{C}_{6} Lh^{\alpha} \sum_{k=1}^{n}(n-k+1)^{\alpha-1} E_{k}  \notag\\  
	&  \leq \widetilde{C}_{7}h^{3} + \widetilde{C}_{8}T^{P}_{n+1} + \widetilde{C}_{9}h^{\alpha} E_{\frac{1}{2}}  +\widetilde{C}_{10}h^{\alpha} \sum_{k=1}^{n}(n-k+1)^{\alpha-1} E_{k}.
	\end{align}
	Consequently, by the discrete Gronwall's inequality (i.e. lemma \ref{lemma24}) the desired result holds.
\end{proof}

\begin{remark}
	Since the global error indicated in the Theorem \ref{theorem24} is dependent on that of the start-up ($E_{\frac{1}{2}},E_{1}$, and $E_{2}$). We suggest employing the start-up scheme described in \textbf{Appendix A} to generate approximate solutions for the first stages (i.e. $\widetilde{y}_{\frac{1}{2}},\widetilde{y}_{1}$, and $\widetilde{y}_{2}$).
\end{remark}

\section{Numerical Illustrations and Simulation}
In this section, we give some numerical experiments through the proposed predictor-corrector scheme (PPC) (\ref{y_corrector})-(\ref{eq216}), the predictor-corrector method introduced by Baleanu-Jajarmi-Hajipour (BJH-PC) \cite{baleanu2018nonlinear}, and the explicit numerical scheme introduced by Toufik-Atangana (TAE) \cite{Toufik2017}, to show the efficiency and accuracy of our new method. The experimental order of convergence (EOC) is computed by
\begin{equation}
EOC= log_{2} \frac{AE\left(\frac{N}{2}\right)}{AE\left(N\right)},
\end{equation}
where $AE$ is the absolute error, which takes the following form:
\begin{equation}
AE:=AE(N)=\underset{1\leqslant k\leqslant N}{\max}\left| y\left(t_{k}\right)-\widetilde{y}_{k} \right|.
\end{equation}
%In addition, we use the acronyms: proposed predictor-corrector scheme (PPC), Baleanu-Jajarmi-Hajipour predictor-corrector scheme (BJH-PC), and Toufik-Atangana explicit scheme (TAE).

% Example 1
\begin{example}\label{Example1}
	Consider the following fractional Initial value problem (IVP):
	\begin{equation}\label{Eq_Example1}
	\left\lbrace \begin{matrix}
	{}^{ABC}_{\quad \,0}D^{\alpha}_{t} y\left(t\right)= t^{n},   &       0< t \leq 2, \\ \\
	y\left(0\right)= 1, &
	\end{matrix}\right. 
	\end{equation}
	where $\alpha\in(0,1)$, and $n\in\mathbb{N}$. On account of \cite{baleanu2018nonlinear}, the exact solution of (\ref{Eq_Example1}) is given by 
	\begin{align}
	&y\left(t\right)=1+\frac{1-\alpha}{AB\left(\alpha\right)}t^{n}+\frac{\alpha \Gamma\left(n+1\right)}{AB\left(\alpha\right)\Gamma\left(\alpha+n+1\right)}t^{\alpha+n}.
	\end{align}
	In addition, the problem \eqref{Eq_Example1} is numerically solved.  
	Table \ref{Ex1_Table1} and table \ref{Ex1_Table2} show the comparison of absolute error of different numerical methods ((PPC) (\ref{y_corrector})-(\ref{eq216}), BJH-PC \cite{baleanu2018nonlinear}, and TAE \cite{Toufik2017}) for \eqref{Eq_Example1} with $n=2,3$, $t \in \left[0,2\right]$ and various values of fractional order $\alpha\in \left\lbrace 0.5,0.7,0.9,0.99\right\rbrace $, where we notice that our method is superior to them in terms of accuracy (i.e. PPC (\ref{y_corrector})-(\ref{eq216}) achieves a lower error than TAE and BJH-PC). Moreover, table \ref{Ex1_Table1} and table \ref{Ex1_Table2} offer that EOC for TAE are roughly $1$, for BJH-PC are roughly $2$, and for PPC (\ref{y_corrector})-(\ref{eq216}) are roughly $3$. Also, we note that the approximate solutions obtained by our proposed scheme get closer to the exact solutions by the increase in $\alpha$. From Figure \ref{Fig1_Example1} and Figure \ref{Fig2_Example1}, we note that the approximate solution obtained by PPC (\ref{y_corrector})-(\ref{eq216}) almost matches with the exact solution with small step size compared to its counterparts. 
	These figures and tables, indicate the efficacy of the current predictor-corrector numerical method (PPC) (\ref{y_corrector})-(\ref{eq216}).
	
	\begin{table}
		\centering
		\caption{The absoluate error of various numerical methods for problem (\ref{Eq_Example1}) with $n=2$, N=$40$, and various values of $\alpha$.}
		\begin{tabular}{|l|l|l|l|l|}
			\hline\noalign{\smallskip}
			Methods   & \qquad $\alpha=0.5$ & \qquad $\alpha=0.7$ & \qquad $\alpha=0.9$ & \qquad $\alpha=0.99$\\
			\noalign{\smallskip}
			\hline
			\noalign{\smallskip}
			PPC  (\ref{y_corrector})-(\ref{eq216})
			& 
			\qquad $5.3 e{-15}$
			& 
			\qquad $1.8 e{-15}$
			& 
			\qquad $3.6 e{-15}$
			& 
			\qquad $9.0 e{-16}$
			\parbox[t]{1cm}{
				\raggedright}\\
			BJH-PC \cite{baleanu2018nonlinear}
			& 
			\qquad $4.1 e{-4}$
			& 
			\qquad $6.2 e{-4}$
			& 
			\qquad $7.7 e{-4}$
			& 
			\qquad $8.3 e{-4}$
			\parbox[t]{1cm}{
				\raggedright}\\
			TAE \cite{Toufik2017} 
			& 
			\qquad $1.3 e{-1}$
			& 
			\qquad $7.4 e{-2}$
			& 
			\qquad $2.5 e{-2}$
			& 
			\qquad $6.1 e{-3}$
			\parbox[t]{1cm}{
				\raggedright}\\						
			\hline
		\end{tabular}
	\end{table}\label{Ex1_Table1}
	
	\begin{table}
		\centering
		\caption{The absolute error, experimental order of convergence, and CPU time in seconds (CTs) of various numerical methods for problem (\ref{Eq_Example1}) with $n=3$.}
		\begin{tabular}{|l | c|ccc|ccc|ccc|}
			\hline
			\multicolumn{1}{|c}{Methods} & \multicolumn{1}{|c|}{}  & \multicolumn{3}{c|}{$\alpha=0.5$} & \multicolumn{3}{c|}{$\alpha=0.7$} & \multicolumn{3}{c|}{$\alpha=0.9$} \\
			&N& AE  & EOC  & CTs  & AE  & EOC  & CTs   & AE  & EOC  & CTs   \\
			\hline
			PPC  (\ref{y_corrector})-(\ref{eq216})
			&$10 $ & $1.8e{-3}$ & $-$ & $3.8e{-2}$  
			&$2.7e{-3}$  & $-$ &$2.4e{-3}$
			&$3.4e{-3}$  & $-$ & $2.2e{-3}$   \\
			&$20$&$2.4e{-4}$ & $2.94$ &$7.7e{-3}$   &$3.6e{-4}$ & $2.94$ &$7.3e{-3}$
			& $4.4e{-4}$  & $2.93$ & $6.4e{-3}$  \\
			&$40$&$3.1e{-5}$&$2.97$ &$2.0e{-2}$   
			& $4.6e{-5}$ & $2.97$& $2.1e{-2}$ 
			& $5.7e{-5}$  & $2.97$ & $2.1e{-2}$ \\
			&$80$ & $3.9e{-6}$& $2.98$&$7.5e{-2}$   
			& $5.8e{-6}$  &$2.99$ & $7.6e{-2}$ 
			& $7.2e{-6}$  & $2.98$ & $7.5e{-2}$  \\
			&$160$&$4.9e{-7}$&$2.99$&$0.31$   
			& $7.2e{-7}$  &$2.99$ & $0.29$ 
			& $9.0e{-7}$  & $2.99$ & $0.29$\\
			&$320$  & $6.2e{-8}$& $2.99$ &$1.12$   
			& $9.1e{-8}$  &$3.00$ & $1.14$ 
			& $1.1e{-7}$  & $3.00$ & $1.15$\\
			\hline
			BJH-PC \cite{baleanu2018nonlinear}
			&$10 $  & $2.4e{-2}$ & $-$ & $2.7e{-2}$   
			& $3.4e{-2}$ & $-$ & $1.8e{-3}$ 
			& $3.9e{-2}$ & $-$ & $9.0e{-4}$   \\
			&$20$& $6.3e{-3}$ & $1.95$ & $3.5e{-3}$ 
			& $8.6e{-3}$ & $1.98$ & $1.7e{-3}$ 
			& $9.7e{-3}$ & $2.00$ & $1.9e{-3}$   \\
			&$40$ &$1.6e{-3}$ & $1.97$ & $3.7e{-3}$ 
			&  $2.2e{-3}$ & $1.99$ & $3.9e{-3}$ 
			& $2.4e{-3}$ & $2.00$ & $4.0e{-3}$   \\
			&$80$ &$4.1e{-4}$& $1.98$ & $9.3e{-3}$  
			& $5.4e{-4}$ & $1.99$ & $1.0e{-2}$  
			& $6.1e{-4}$ & $2.00$ & $1.0e{-2}$   \\
			&$160$ &$1.0e{-4}$& $1.99$& $2.6e{-2}$  
			& $1.4e{-4}$ & $2.00$ & $3.1e{-2}$  
			& $1.5e{-4}$ & $2.00$ & $3.1e{-2}$   \\
			&$320$ & $2.6e{-5}$&$1.99$& $8.4e{-2}$  
			& $3.4e{-5}$ & $2.00$ & $0.10$  
			& $3.8e{-5}$ & $2.00$ & $0.11$   \\
			\hline
			TAE \cite{Toufik2017}					
			&10 & $1.5e{-0}$ & $-$ & $9.6e{-3}$  
			& $9.5e{-1}$  & $-$ & $2.6e{-3}$
			& $4.1e{-1}$  & $-$ & $1.3e{-3}$   \\
			&20 & $7.7e{-1}$ & $1.01$ & $4.7e{-3}$   
			& $4.5e{-1}$ & $1.07$ & $4.6e{-3}$
			& $1.7e{-1}$ & $1.28$ & $4.8e{-3}$  \\
			&40  & $3.9e{-1}$& $1.00$ & $1.7e{-2}$  
			& $2.2e{-1}$ & $1.04$& $1.8e{-2}$ 
			& $7.4e{-2}$  & $1.18$ & $1.8e{-2}$ \\
			&80  & $1.9e{-1}$& $1.00$ & $8.4e{-2}$  
			& $1.1e{-1}$  &$1.02$ & $7.0e{-2}$ 
			& $3.4e{-2}$ & $1.10$ & $7.0e{-2}$\\
			&160  & $9.6e{-2}$& $1.00$ & $0.27$  
			& $5.4e{-2}$  &$1.01$ & $0.33$ 
			& $1.7e{-2}$  & $1.06$ & $0.29$\\
			&320    & $4.8e{-2}$& $1.00$ & $1.11$   
			& $2.7e{-2}$ & $1.00$ & $1.13$ 
			& $8.1e{-3}$& $1.03$ & $1.15$\\
			\hline
		\end{tabular}
	\end{table}\label{Ex1_Table2}
	
	\begin{figure}[tbp]
		%\centering
		\includegraphics[width=6.5in]{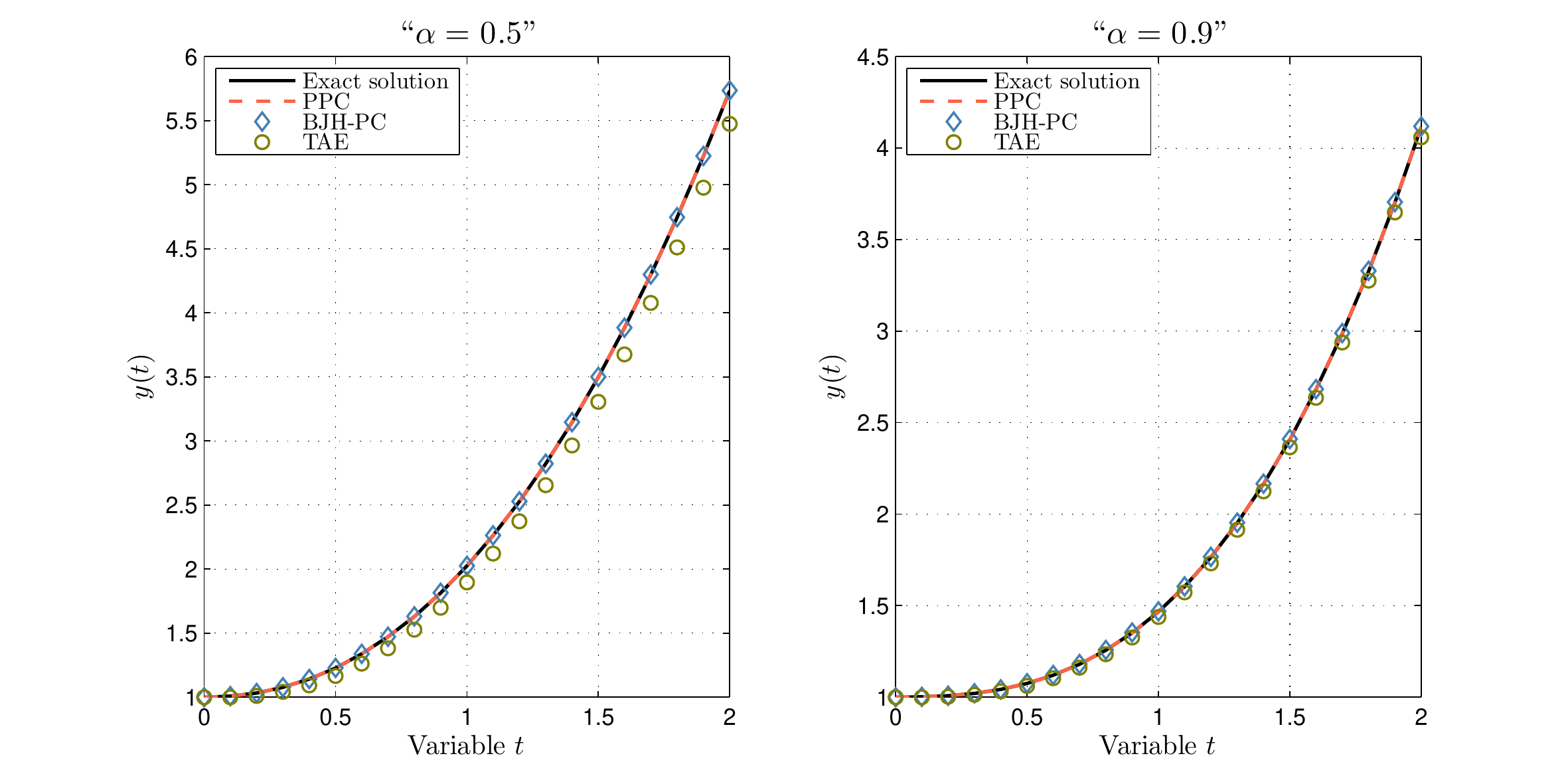}
		\caption{Comparison of the exact and the numerical solutions of problem (\ref{Eq_Example1}) with $n=2$, and $N=10$.}
		\label{Fig1_Example1}
	\end{figure}
	
	\begin{figure}[tbp]
		\centering
		\includegraphics[width=6.5in]{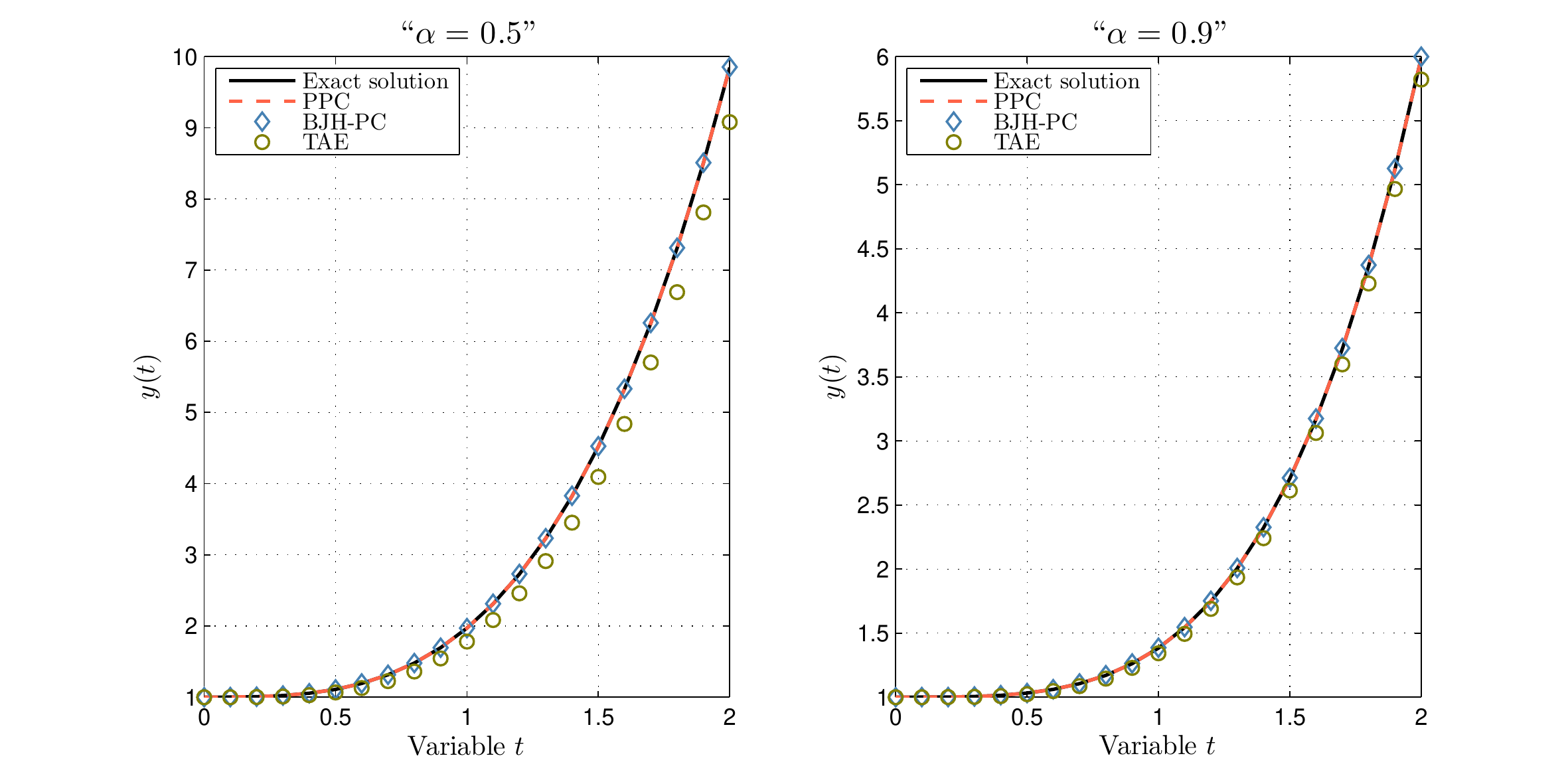}
		\caption{Comparison of the exact and the numerical solutions of problem (\ref{Eq_Example1}) with $n=3$, and $N=10$.}
		\label{Fig2_Example1}
	\end{figure}
	
\end{example}

% Example 2
\begin{example}\label{Example2}
	Consider the following Atangana-Beleanu-fractional differential equation:
	\begin{equation}\label{Eq_Example2}
	\left\lbrace \begin{matrix}
	{}^{ABC}_{\quad \,0}D^{\alpha}_{t} y\left(t\right)= t-y(t),   &       0< t \leq 1, \\ \\
	y\left(0\right)= 0, &
	\end{matrix}\right. 
	\end{equation}
	where $\alpha\in(0,1)$. According to \cite{baleanu2018nonlinear}, the exact solution of (\ref{Eq_Example2}) is given by 
	\begin{equation}\label{ExactS_Exm2}
	y\left(t\right)=\displaystyle
	\frac{1}{AB\left(\alpha\right)+1-\alpha} \left(  \left(1-\alpha\right)t   \textbf{E}_{\alpha,2}\left( -\frac{\alpha}{AB\left(\alpha\right)+1-\alpha} t^{\alpha}\right) +\alpha t^{\alpha+1}    \textbf{E}_{\alpha,\alpha+2}\left(- \frac{\alpha}{AB\left(\alpha\right)+1-\alpha} t^{\alpha}\right)  \right),
	\end{equation}
	with $ \textbf{E}_{\alpha,\beta}(.)$ denotes the Mittag-Leffler function of two parameters $\alpha$ and $\beta$, which defined by
	\begin{equation}
	\begin{matrix}
	\textbf{E}_{\alpha,\beta}(z)=\displaystyle\sum^{\infty}_{k=0}\frac{{z}^{k}}{\Gamma (\alpha k+1)}, & z,\alpha,\beta\in \mathbb{C},& \text{and }Re(\alpha),Re(\beta)>0,
	\end{matrix}
	\end{equation}
	and the exact solution \eqref{ExactS_Exm2} is calculated using the algorithm $mlf.m$ (see \cite{PodlubnyMlf}) evaluated with accuracy $10^{-12}$. Furthermore, the problem \eqref{Eq_Example2} is numerically solved.  
	Table \ref{Ex2_Table1} and table \ref{Ex2_Table2} show the comparison of absolute error of different numerical methods ((PPC) (\ref{y_corrector})-(\ref{eq216}), BJH-PC \cite{baleanu2018nonlinear}, and TAE \cite{Toufik2017}) for \eqref{Eq_Example2} with $AB(\alpha)=1$, $AB(\alpha)=1-\alpha+\frac{\alpha}{\Gamma(\alpha)}$, $t \in \left[0,1\right]$ and various values of fractional order $\alpha\in \left\lbrace 0.5,0.55,0.7,0.9,0.95\right\rbrace $, where we notice that our method is superior to them in terms of accuracy (i.e. PPC (\ref{y_corrector})-(\ref{eq216}) achieves a lower error than TAE and BJH-PC). From Figure \ref{Fig1_Example2} and Figure \ref{Fig2_Example2}, we note that the approximate solution obtained by PPC (\ref{y_corrector})-(\ref{eq216}) almost matches with the exact solution with small step size compared to its counterparts.
	These figures and tables, again confirm the efficacy of the current predictor-corrector numerical method (PPC) (\ref{y_corrector})-(\ref{eq216}).

	\begin{table}
		\centering
		\caption{The absolute error, experimental order of convergence, and CPU time in seconds (CTs) of various numerical methods for problem (\ref{Eq_Example2}) with $AB(\alpha)=1$.}
		\begin{tabular}{|l | c|ccc|ccc|ccc|}
			\hline
			\multicolumn{1}{|c}{Methods} & \multicolumn{1}{|c|}{}  & \multicolumn{3}{c|}{$\alpha=0.55$} & \multicolumn{3}{c|}{$\alpha=0.75$} & \multicolumn{3}{c|}{$\alpha=0.95$} \\
			&N& AE  & EOC  & CTs  & AE  & EOC  & CTs   & AE  & EOC  & CTs   \\
			\hline
			PPC  (\ref{y_corrector})-(\ref{eq216})
			&10  & $1.5e{-3}$ & $-$ & $3.4e{-2}$  
			& $3.5e{-4}$  & $-$ & $1.1e{-3}$
			& $2.7e{-5}$  & $-$ & $6.0e{-4}$   \\
			&20  & $7.4e{-4}$ & $1.03$ & $2.9e{-3}$   
			& $8.8e{-5}$ & $1.99$ & $1.1e{-3}$
			& $3.2e{-6}$ & $3.05$ & $1.1e{-3}$  \\
			&40  & $3.7e{-4}$& $0.99$ & $3.5e{-3}$   
			& $2.3e{-5}$ & $1.93$ & $3.0e{-3}$ 
			& $4.4e{-7}$ & $2.89$ & $3.0e{-3}$   \\
			&80  & $1.8e{-4}$& $1.03$ & $1.0e{-2}$   
			& $6.1e{-6}$  &$1.91$ & $1.0e{-2}$ 
			& $6.7e{-8}$  & $2.70$& $1.0e{-2}$    \\
			&160 & $9.0e{-5}$& $1.03$ & $3.8e{-2}$   
			& $2.7e{-6}$  &$1.21$ & $3.7e{-2}$ 
			& $1.2e{-8}$  &$2.51$ & $3.7e{-2}$    \\
			&320 & $4.4e{-5}$& $1.03$ & $0.14$   
			& $1.3e{-6}$  &$1.03$ & $0.14$ 
			& $2.3e{-9}$  &$2.33$ & $0.14$    \\
			\hline
			BJH-PC \cite{baleanu2018nonlinear}
			&10  & $2.6e{-2}$ & $-$ & $1.7e{-2}$  
			& $8.9e{-3}$ & $-$ & $8.0e{-4}$ 
			& $1.8e{-3}$ & $-$ & $1.0e{-3}$   \\
			&20 & $1.2e{-2}$ & $1.18$ & $3.3e{-3}$ 
			& $3.7e{-3}$ & $1.26$ & $1.5e{-3}$    
			& $6.7e{-4}$ & $1.41$ & $1.7e{-3}$   \\
			&40 & $5.3e{-3}$ & $1.13$ & $3.4e{-3}$ 
			& $1.7e{-3}$ & $1.16$ & $3.3e{-3}$ 
			& $2.8e{-4}$ & $1.25$ & $3.3e{-3}$   \\
			&80 & $2.5e{-3}$ & $1.09$ & $8.2e{-3}$  
			& $7.8e{-4}$ & $1.09$ & $8.3e{-3}$  
			& $1.3e{-4}$ & $1.15$ & $8.2e{-3}$   \\
			&160& $1.2e{-3}$ & $1.06$ & $2.5e{-2}$  
			& $3.8e{-4}$ & $1.06$ & $2.4e{-2}$  
			& $6.0e{-5}$ & $1.08$ & $2.4e{-2}$   \\
			&320& $5.8e{-4}$ & $1.04$ & $7.9e{-2}$  
			& $1.8e{-4}$ & $1.03$ & $7.9e{-2}$  
			& $2.9e{-5}$ & $1.04$ & $7.7e{-2}$   \\
			\hline
			TAE \cite{Toufik2017}					&10  & $3.6e{-2}$ & $-$ & $5.6e{-3}$  
			& $2.5e{-2}$  & $-$ & $7.0e{-4}$
			& $9.6e{-3}$  & $-$ & $6.0e{-4}$   \\
			&20  & $1.7e{-2}$ & $1.06$ & $2.5e{-3}$   
			& $1.2e{-2}$ & $1.12$ & $2.2e{-3}$
			& $3.7e{-3}$  & $1.40$ & $2.2e{-3}$  \\
			&40  & $8.4e{-3}$& $1.05$ & $8.8e{-3}$   
			& $5.5e{-3}$ & $1.08$& $8.6e{-3}$ 
			& $1.5e{-3}$ & $1.26$ & $8.4e{-3}$ \\
			&80  & $4.1e{-3}$& $1.03$ & $3.6e{-2}$   
			& $2.6e{-3}$  & $1.05$ & $3.4e{-2}$ 
			& $6.8e{-4}$  & $1.16$ & $3.4e{-2}$\\
			&160 & $2.0e{-3}$& $1.02$ & $0.14$   
			& $1.3e{-3}$  & $1.03$ & $0.13$ 
			& $3.2e{-4}$  & $1.09$ & $0.14$\\
			&320 & $9.9e{-4}$& $1.02$ & $0.55$   
			& $6.4e{-4}$  & $1.02$ & $0.54$ 
			& $1.5e{-4}$  & $1.05$ & $0.55$\\
			\hline
		\end{tabular}
	\end{table}\label{Ex2_Table1}
	
	\begin{table}
		\centering
		\caption{The absolute error, experimental order of convergence, and CPU time in seconds (CTs) of various numerical methods for problem (\ref{Eq_Example2}) with $AB(\alpha)=1-\alpha+\frac{\alpha}{\Gamma(\alpha)}$.}
		\begin{tabular}{|l | c|ccc|ccc|ccc|}
			\hline
			\multicolumn{1}{|c}{Methods} & \multicolumn{1}{|c|}{}  & \multicolumn{3}{c|}{$\alpha=0.55$} & \multicolumn{3}{c|}{$\alpha=0.75$} & \multicolumn{3}{c|}{$\alpha=0.95$} \\
			&N& AE  & EOC  & CTs  & AE  & EOC  & CTs   & AE  & EOC  & CTs   \\
			\hline
			PPC  (\ref{y_corrector})-(\ref{eq216})
			&10  & $2.3e{-2}$ & $-$ & $3.4e{-2}$  
			& $5.5e{-4}$  & $-$ & $1.1e{-3}$
			& $2.9e{-5}$  & $-$ & $6.0e{-4}$   \\
			&20  & $9.7e{-3}$ & $1.25$ & $2.7e{-3}$   
			& $1.3e{-4}$ & $2.03$ & $1.2e{-3}$
			& $3.5e{-6}$  & $3.05$ & $1.1e{-3}$  \\
			&40  & $4.3e{-3}$& $1.17$ & $3.5e{-3}$   
			& $3.4e{-5}$ & $1.99$ & $3.0e{-3}$ 
			& $4.7e{-7}$ & $2.88$ & $3.1e{-3}$ \\
			&80    & $2.0e{-3}$& $1.12$ & $1.1e{-2}$   
			& $1.2e{-5}$  & $1.45$ & $1.0e{-2}$ 
			& $7.3e{-8}$  & $2.70$ & $1.0e{-2}$\\
			&160   & $9.3e{-4}$& $1.09$ & $3.8e{-2}$   
			& $6.1e{-6}$  &$1.01$ & $3.7e{-2}$ 
			& $1.3e{-8}$  &$2.51$ & $3.7e{-2}$\\
			&320   & $4.5e{-4}$& $1.06$ & $0.14$   
			& $3.0e{-6}$  & $1.01$ & $0.14$ 
			& $2.6e{-9}$  & $2.33$ & $0.14$\\
			\hline
			BJH-PC \cite{baleanu2018nonlinear}
			&10  & $4.9e{-2}$ & $-$ & $1.6e{-2}$  
			& $1.2e{-2}$ & $-$ & $9.0e{-4}$ 
			& $1.9e{-3}$ & $-$ & $8.0e{-4}$   \\
			&20 & $2.1e{-2}$ & $1.23$ & $3.2e{-3}$ 
			& $4.9e{-3}$ & $1.26$ & $1.5e{-3}$    
			& $6.9e{-4}$ & $1.41$ & $1.5e{-3}$   \\
			&40 & $9.5e{-3}$ & $1.15$ & $3.3e{-3}$ 
			&  $2.2e{-3}$ & $1.17$ & $3.4e{-3}$ 
			& $2.9e{-4}$  & $1.25$ & $3.2e{-3}$   \\
			&80 & $4.4e{-3}$ & $1.10$ & $8.2e{-3}$  
			& $1.0e{-3}$ & $1.11$ & $8.3e{-3}$  
			& $1.3e{-4}$ & $1.15$ & $8.3e{-3}$   \\
			&160& $2.1e{-3}$ & $1.07$ & $2.4e{-2}$  
			& $4.8e{-4}$ & $1.07$ & $2.4e{-2}$  
			& $6.2e{-5}$ & $1.08$ & $2.4e{-2}$   \\
			&320& $1.0e{-3}$ & $1.05$ & $7.8e{-2}$  
			& $2.4e{-4}$ & $1.04$ & $7.8e{-2}$  
			& $3.0e{-5}$ & $1.04$ & $8.5e{-2}$   \\
			\hline
			TAE \cite{Toufik2017}					
			&10  & $4.2e{-2}$ & $-$ & $5.5e{-3}$  
			& $2.8e{-2}$  & $-$ & $6.0e{-4}$
			& $9.9e{-3}$  & $-$ & $6.0e{-4}$   \\
			&20  & $2.0e{-2}$ & $1.06$ & $2.5e{-3}$   
			& $1.3e{-2}$ & $1.11$ & $2.2e{-3}$
			& $3.8e{-3}$  & $1.39$ & $2.3e{-3}$  \\
			&40  & $9.7e{-3}$& $1.04$ & $8.8e{-3}$   
			& $6.1e{-3}$ & $1.08$& $8.5e{-3}$ 
			& $1.6e{-3}$ & $1.26$ & $8.7e{-3}$ \\
			&80  & $4.8e{-3}$& $1.03$ & $3.5e{-2}$   
			& $3.0e{-3}$  & $1.05$ & $3.4e{-2}$ 
			& $7.0e{-4}$  & $1.16$ & $3.4e{-2}$  \\
			&160 & $2.3e{-3}$& $1.02$ & $0.14$   
			& $1.4e{-3}$  & $1.03$ & $0.14$ 
			& $3.3e{-4}$  & $1.09$ & $0.14$  \\
			&320 & $1.2e{-3}$& $1.02$ & $0.56$   
			& $7.2e{-4}$  & $1.02$ & $0.55$ 
			& $1.6e{-4}$  & $1.05$ & $0.57$  \\
			\hline
		\end{tabular}
	\end{table}\label{Ex2_Table2}
	
	\begin{figure}[tbp]
		%\centering
		\includegraphics[width=6.5in]{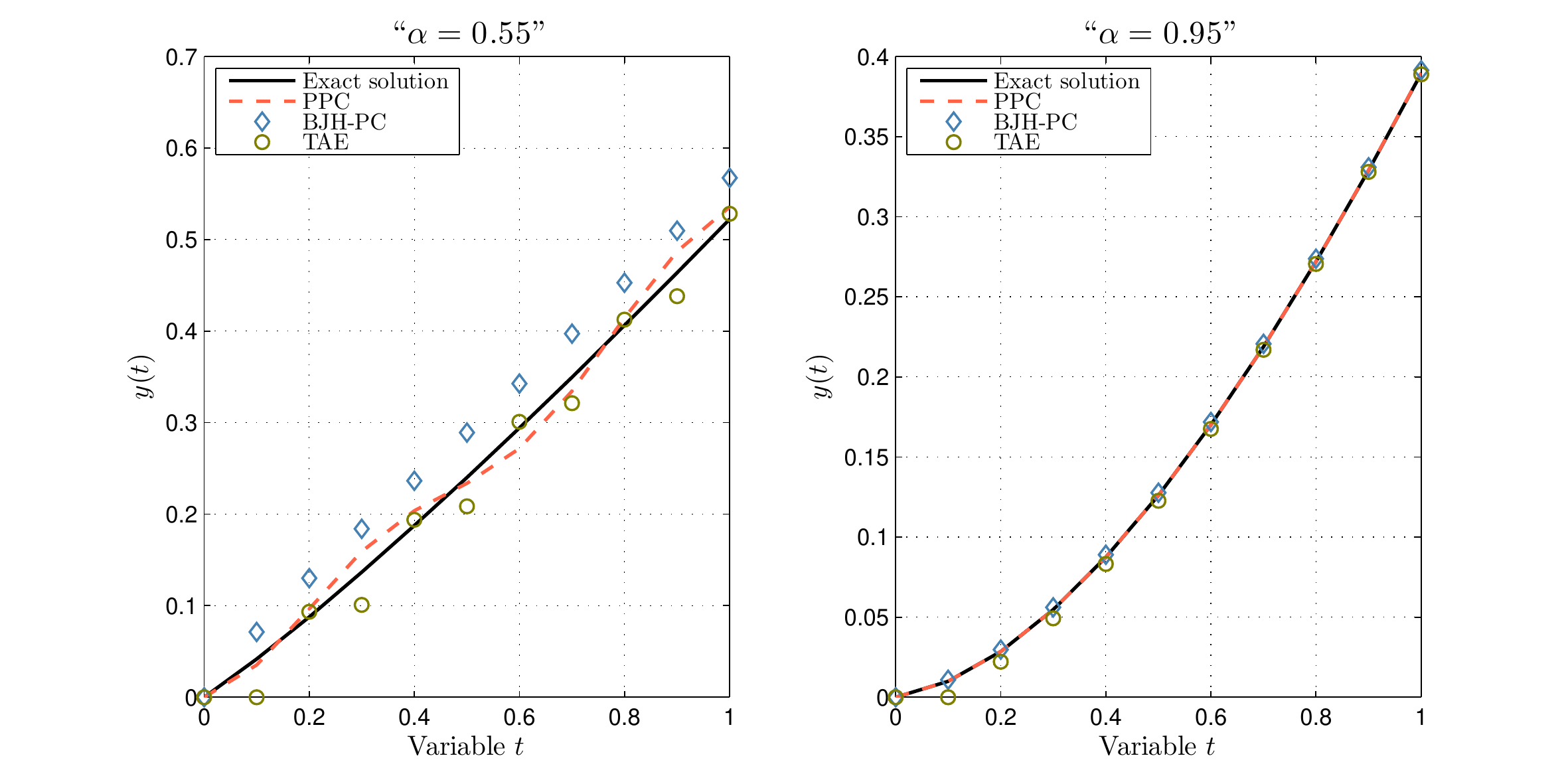}
		\caption{Comparison of the exact and the numerical solutions of problem (\ref{Eq_Example2}) with $N=10$, and $AB(\alpha)=1-\alpha+\frac{\alpha}{\Gamma(\alpha)}$.}
		\label{Fig1_Example2}
	\end{figure}
	
	\begin{figure}[tbp]
		\centering
		\includegraphics[width=6.5in]{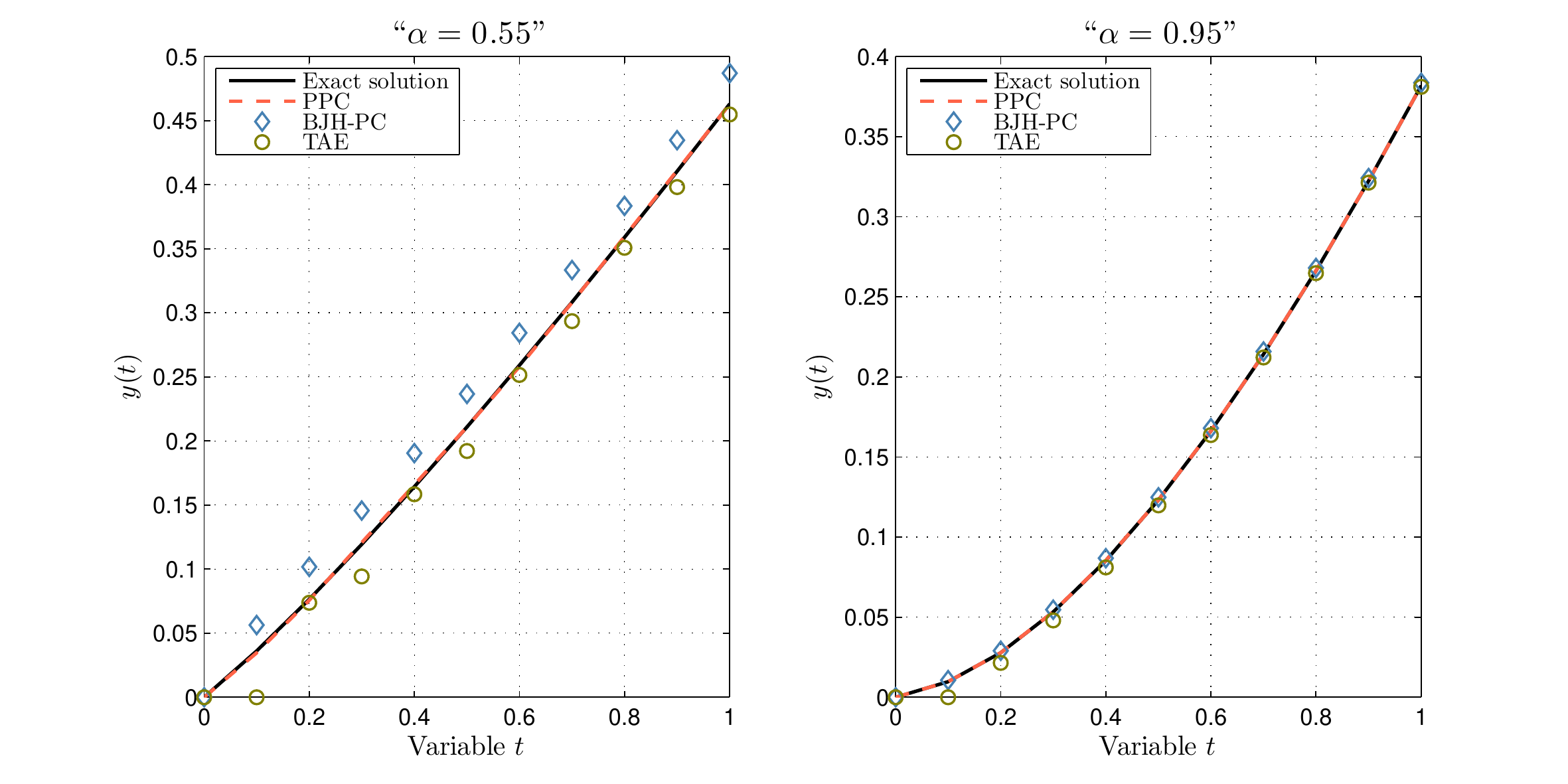}
		\caption{Comparison of the exact and the numerical solutions of problem (\ref{Eq_Example2}) with $N=10$, and $AB(\alpha)=1$.}
		\label{Fig2_Example2}
	\end{figure}
	
\end{example}

\begin{example}\textbf{(Dynamics of a generalized susceptible-infectious model with ABC-fractional derivative)}\\
	In the previous examples, we have considered some differential equations with known exact solutions. Let us now analyze a realistic ABC-fractional epidemic model using qualitative analysis (Lyapunov theory) of solutions and validate the theoretical results numerically by means of the proposed predictor-corrector method. We consider the epidemic phenomenon proposed in \cite{Djebara2022}, which is an extended version of the SI epidemic model with a nonlinear incidence $\varphi$, and we include the fractional derivative which involving Mittag-Leffer kernel. Thus, the novel system is described as follows:
	\begin{equation}
	\label{eq31}
	\begin{cases}
	{}^{ABC}_{\quad \,0}D^{\alpha}_{t}u(t)= \Lambda - \gamma u(t)  \varphi(v(t)) -\mu u(t) ,&t>0, \\
	{}^{ABC}_{\quad \,0}D^{\alpha}_{t}v(t)= \gamma u(t)  \varphi(v(t))-\sigma v(t) ,&t>0,\\
	u(0)=u_{0}>0, \ v(0) = v_{0}>0,
	\end{cases}
	\end{equation}
	where $u(t)$, $v(t)$ represent the numbers of susceptible individuals, and infective individuals at time t, respectively. $\Lambda$ is the birth rate of the population, $\gamma$ is the disease transmission rate, $\mu$ is the natural death rate, and $\sigma:= \widetilde{\sigma}+ \mu$, with  $\widetilde{\sigma}$ is the death rate due to disease, which all are positive real numbers, and $\alpha\in(0,1)$. The incidence function $\varphi(v)$ introduces a nonlinear relation between the susceptible individuals and infective individuals. We suppose that $\varphi\in C^{1}(\mathbb{R}_{+};\mathbb{R}_{+})$, and satisfies:
	\begin{equation}
	\label{eq32}
	\begin{cases}
	\varphi(0)=0 ,\\
	0<v \varphi^{\prime}(v)\leq \varphi(v),& \forall v>0.
	\end{cases}
	\end{equation}
	According to the above condition and Lemma 1 in \cite{abdelmalek2017}, it is classical task to prove the existence of unique solution for system (\ref{eq31}) (see e.g. \cite{Ucar2021,Sintunavarat2022}). As well as, the system (\ref{eq31}) has a disease-free equilibrium $E_0=(u^{0},0)$ with $u^{0}:=\displaystyle\frac{\Lambda}{\mu}$, and a unique endemic equilibrium $E^*:= (u^{*},v^{*})$ if $R_0:=\displaystyle\frac{\Lambda\gamma }{\mu \sigma} \varphi^{\prime}(0)>1$ (cf. \cite{Djebara2022}). In the remainder, the following notation $\mathbb{R}_{+}:=\left[0,+\infty \right) $ will be useful. The feasible region of the suggested model (\ref{eq31}) is given by
	\begin{equation}\label{F_region}
	\Upsilon:=\left\lbrace (x_{1},x_{2}) \in \mathbb{R}^{2}_{+}: x_{1}+x_{2}\leq \frac{\Lambda}{\mu}\right\rbrace .
	\end{equation}
	Now, we investigate the global asymptotic stability of the equilibrium points $E_{0}$ and $E^{*}$, which is determined by the reproduction number $R_{0}$, as a result, the cases $R_{0}\leqslant 1$ and $R_{0}>1$ are treated independently.\\
	Firstly, we provide the global stability result of the disease-free equilibrium $E_0$.
	\begin{theorem}\label{Th1SI_Stability}
		Let $\alpha\in(0,1)$ and $R_{0}\leqslant1$. Assume that (\ref{eq32}) holds. Then, the disease-free equilibrium $E_0$ is globally asympotatically stable in the feasible region $\Upsilon$.
	\end{theorem}
	\begin{proof}
		We begin by rewriting the first equation of the system \eqref{eq31}, with taking into account this relationship $\Lambda=\mu u^{0}$, such that 
		\begin{equation}
		\label{eq34}
		{}^{ABC}_{\quad \,0}D^{\alpha}_{t}u(t)= -\Big(\mu +  \gamma \varphi(v)\Big)(u-u_{0})-\gamma \varphi(v) u_{0} 
		\end{equation}
		We consider the following candidate Lyapunov function:
		\begin{equation} 
		V:  \left\lbrace (u,v) \in \Upsilon : u>0 \right\rbrace  \longrightarrow \mathbb{R}_{+},
		\end{equation}
		such that
		\begin{equation}
		\label{eq35}
		V(u,v)= \frac{(u-u^{0})^{2}}{2u^{0}}+v.
		\end{equation}
		Thanks to the property of ABC-fractional derivative given in Lemma 3.1 from \cite{Hernandez2020}, we get
		\begin{equation}
		\label{eq36}
		\begin{split}
		{}^{ABC}_{\quad \,0}D^{\alpha}_{t}V \leq \frac{(u-u^{0})}{u^{0}} {}^{ABC}_{\quad \,0}D^{\alpha}_{t}u(t) +{}^{ABC}_{\quad \,0}D^{\alpha}_{t} v(t) .
		\end{split}
		\end{equation}
		Then, by employing \eqref{eq31}-\eqref{F_region}, and Lemma 1 \cite{abdelmalek2017}, we obtain
		\begin{align}
		{}^{ABC}_{\quad \,0}D^{\alpha}_{t} V  \leq &  -\left( \mu +\gamma  \varphi(v)\right)   \frac{(u-u^{0})^{2}}{u^{0}} - \gamma u^{0} \varphi(v)  + \gamma u  \varphi(v) - \sigma v, \notag\\
		\leq &  -\left( \mu +\gamma  \varphi(v)\right)   \frac{(u-u^{0})^{2}}{u^{0}} + \gamma \varphi^{\prime}(0)u^{0}v- \sigma v , \notag\\  
		=&-\left( \mu +\gamma  \varphi(v)\right)   \frac{(u-u^{0})^{2}}{u^{0}}+ \left( \frac{\gamma \Lambda  \varphi^{\prime}(0)}{\mu} - \sigma \right) v, \notag \\
		=&-\left( \mu +\gamma  \varphi(v)\right)   \frac{(u-u^{0})^{2}}{u^{0}}-\sigma\left(1- R_0 \right) v.
		\end{align}
		Therefore, $R_{0}\leqslant1$, guarantees that ${}^{ABC}_{\quad \,0}D^{\alpha}_{t} V  \leq0$ for $\alpha\in(0,1)$. Consequently, by the Lyapunov stability theorem (see Theorem 2.7 from \cite{Hernandez2020} and \cite{WeiLyaponuv2022}), we conclude that the disease-free equilibrium $E^0$ is globally asymptotically stable in $\Upsilon$.
	\end{proof}
	
	Secondly, we provide the global stability result of the endemic equilibrium $E^*$. Before do that, we state a lemma which will be useful (see Lemma 2 from \cite{Djebara2022}):
	\begin{lemma} \label{lemma_F_phi}
		Assume that the function $\varphi$ satisfied  \eqref{eq32}  and  
		\begin{equation}\label{eq39}
		\begin{matrix}
		F(x)=x-1-ln(x), &   \forall x>0.
		\end{matrix}
		\end{equation}
		We thus get the following inequality
		\begin{equation}\label{eq40}
		\begin{matrix}
		F\left(\displaystyle \frac{\varphi(v)}{\varphi(v^{*})} \right)  \leq F \left( \displaystyle \frac{v}{v^{*}} \right), & \forall v>0,
		\end{matrix}
		\end{equation}
		where $v^{*}$ is the second component of the equilibrium point $E^{*}$.
	\end{lemma} 
	\begin{theorem}\label{Th2SI_Stability}
		Let $\alpha\in(0,1)$ and $R_{0}>1$. Suppose that (\ref{eq32}) holds. Then, the endemic equilibrium $E^{*}$ is globally asymptotically stable in interior of the feasible region $\Upsilon$.
	\end{theorem}
	\begin{proof}
		The first two equation of system (\ref{eq31}) at $E^*$, clearly give
		\begin{equation}
		\begin{cases}
		\Lambda=\mu u^{*}+\gamma u^{*}\varphi(v^{*}) ,\\
		v^{*} \sigma= \gamma u^{*} \varphi(v^{*}).
		\end{cases}
		\end{equation}
		It follows that
		\begin{equation}
		\label{eq41}
		\begin{cases}
		{}^{ABC}_{\quad \,0}D^{\alpha}_{t} u = \mu u^{*} \Big(1-\displaystyle \frac{u}{u^{*}}\Big)+\gamma u^{*}\varphi(v^{*})\left( 1- \displaystyle\frac{u \varphi(v)}{u^{*} \varphi(v^{*})}\right),   \\
		{}^{ABC}_{\quad \,0}D^{\alpha}_{t} v=\gamma u^{*}\varphi(v^{*})\left(\displaystyle \frac{u \varphi(v)}{u^{*} \varphi(v^{*})}-\displaystyle \frac{v}{v^{*}}\right).
		\end{cases}
		\end{equation}
		We consider the following candidate Lyapunov function:
		\begin{equation} 
		L:  \left\lbrace (u,v) \in \Upsilon : u>0, v>0 \right\rbrace  \longrightarrow \mathbb{R}_{+},
		\end{equation}
		such that
		\begin{equation}
		L(u,v)= u^{*} F \left(\displaystyle \frac{u}{u^{*}} \right) + v^{*} F\left(\displaystyle \frac{v}{v^{*}}\right) .
		\end{equation}
		We now apply the property of ABC-fractional derivative given in Lemma 3.2 from \cite{Hernandez2020}, to obtain
		\begin{equation}
		\label{eq43}
		{}^{ABC}_{\quad \,0}D^{\alpha}_{t} L\leq \left(1-\displaystyle \frac{u^{*}}{u} \right)  {}^{ABC}_{\quad \,0}D^{\alpha}_{t} u + \left( 1-\displaystyle \frac{v^{*}}{v}\right)  {}^{ABC}_{\quad \,0}D^{\alpha}_{t} v
		\end{equation}
		From \eqref{eq41} and by immediately computation, we get
		\begin{align}
		{}^{ABC}_{\quad \,0}D^{\alpha}_{t} L & \leq  \mu u^{*} \left( 1- \displaystyle \frac{u}{u^{*}}\right) \left(1-\displaystyle \frac{u^{*}}{u}\right) +\gamma u^{*}\varphi(v^{*})\left(\displaystyle\frac{u \varphi(v)}{u^{*} \varphi(v^{*})}-\frac{v}{v^{*}}\right)\left(1-\displaystyle \frac{v^{*}}{v}\right)  \notag \\
		& \ \ \ +\gamma u^{*}\varphi(v^{*})\left( 1- \frac{u \varphi(v)}{u^{*} \varphi(v^{*})}\right) \left(1-\displaystyle \frac{u^{*}}{u}\right), \notag \\
		& \leq - \mu u^{*} \left( F \Big(\frac{u}{u^{*}}\Big)+F \left( \frac{u^{*}}{u}\right) \right) -\gamma u^{*}\varphi(v^{*})\left( F \left( \frac{u^{*}}{u}\right)  + F\left(  \frac{ v^{*} u \varphi(v)}{u^{*} v \varphi(v^{*})}\right) \right)  \\ \qquad \qquad \qquad 
		& \ \ \ + \gamma u^{*}\varphi(v^{*})\left( F \left(  \frac{\varphi(v)}{\varphi(v^{*})}\right) - F \left( \frac{v}{v^{*}}\right) \right) .
		\end{align}
		Therefore, the nonnegativity of function $F$ (defined in \eqref{eq39}) on $(0,+\infty)$, and the lemma \ref{lemma_F_phi} ensure that ${}^{ABC}_{\quad \,0}D^{\alpha}_{t} L  \leq0$ for $\alpha\in(0,1)$. Consequently, by the Lyapunov stability theorem (see Theorem 2.7 and Theorem 2.8 \cite{Hernandez2020} as well as \cite{WeiLyaponuv2022}), we get the desired result.
	\end{proof}
	
	Next, we present some numerical simulation of the general model \eqref{eq31} through several sets of parameters (see table \ref{Table_parametersSI}) to examine the effectiveness of the proposed predictor-corrector scheme in the current work (i.e. PPC (\ref{y_corrector})-(\ref{eq216})), as well as the feasibility of the theoretical findings in this example.
	\begin{table}[h!]
		\centering
		\caption{Diverse sets of system's \eqref{eq31} parameter values used in numerical simulations.}
		\begin{tabular}{|l|l|l|l|l|l|}
			\hline\noalign{\smallskip}
			Set of parameters   & \qquad $\Lambda$ & \qquad $\gamma$ & \qquad $\mu$ & \qquad $\widetilde{\sigma}$& \qquad $R_{0} \ (\text{with } \varphi^\prime(0)=1)$\\
			\noalign{\smallskip}
			\hline
			\noalign{\smallskip}
			\hfil Set $1$
			& 
			\qquad $0.3$
			& 
			\qquad $0.1$
			& 
			\qquad $0.44$
			& 
			\qquad $0.38$
			& 
			\qquad $0.0831$
			\parbox[t]{1cm}{\raggedright}\\
			\hfil Set $2$
			& 
			\qquad $0.03$
			& 
			\qquad $0.4$
			& 
			\qquad $0.0365$
			& 
			\qquad $0.0135$
			& 
			\qquad $6.5753$
			\parbox[t]{1cm}{\raggedright}\\
			\hfil Set $3$
			& 
			\qquad $0.3$
			& 
			\qquad $0.1$
			& 
			\qquad $0.32$
			& 
			\qquad $0.5$
			& 
			\qquad $0.1143$
			\parbox[t]{1cm}{\raggedright}\\	
			\hfil Set $4$
			& 
			\qquad $0.03$
			& 
			\qquad $0.4$
			& 
			\qquad $0.032$
			& 
			\qquad $0.018$
			& 
			\qquad $7.5$
			\parbox[t]{1cm}{\raggedright}\\					
			\hline
		\end{tabular}
	\end{table}\label{Table_parametersSI}
	
	The following is a description of the results:
	\begin{itemize}
		\item According to the first set of parameters in table \ref{Table_parametersSI}, with $\varphi(v)=v \ \left(\text{resp. } \varphi(v)=\frac{v}{1+0.01v}\right) $, and $( u_{0},v_{0})=\left( 0.52,\frac{\Lambda}{\mu}-0.52\right) $, the Theorem \ref{Th1SI_Stability} guarantees (for $\alpha\in(0,1)$ and any positive initial data belongs to $\Upsilon$) the global asymptotic stability of the disease-free equilibrium $E_{0}=(u_{0},0)=(0.6818,0)$, which is confirmed by the numerical results depicted in Figure \ref{Ex3_Fig1} (resp. Figure \ref{Ex3_Fig3}) with various values of fractional order $\alpha\in\left\lbrace 0.8,0.85,0.9,0.99 \right\rbrace $.  \\
		\item According to the second set of parameters in table \ref{Table_parametersSI}, with $\varphi(v)=v \ \left(\text{resp. } \varphi(v)=\frac{v}{1+0.01v}\right) $, where $(u_{0},v_{0})=\left( 0.6,\frac{\Lambda}{\mu}-0.6\right) $, the Theorem \ref{Th2SI_Stability} guarantees (for $\alpha\in(0,1)$ and any positive initial data belongs to $\Upsilon$) the global asymptotic stability of the endemic equilibrium $E^{*}=(u^{*},v^{*})=(0.125,0.5087)$ (resp. $E^{*}=(0.1256,0.5083)$), which is confirmed by the numerical results depicted in Figure \ref{Ex3_Fig2} (resp. Figure \ref{Ex3_Fig4}) with various values of fractional order $\alpha\in\left\lbrace 0.8,0.85,0.9,0.99 \right\rbrace $.  \\
		\item According to the third set of parameters in table \ref{Table_parametersSI}, with $\varphi(v)=v \ \left(\text{resp. } \varphi(v)=\frac{v}{1+0.01v}\right) $, and $\alpha=0.99$, the Theorem \ref{Th1SI_Stability} guarantees (for any positive initial data belongs to $\Upsilon$) the global asymptotic stability of the disease-free equilibrium $E_{0}=(u_{0},0)=(0.9375,0)$, which is confirmed by the numerical results depicted in Figure \ref{Ex3_Fig5} (resp. Figure \ref{Ex3_Fig7}) with various values of initial data. \\
		\item According to the fourth set of parameters in table \ref{Table_parametersSI}, with $\varphi(v)=v \ \left(\text{resp. } \varphi(v)=\frac{v}{1+0.01v}\right) $, and $\alpha=0.99$, the Theorem \ref{Th2SI_Stability} guarantees (for any positive initial data belongs to $\Upsilon$) the global asymptotic stability of the endemic equilibrium $E^{*}=(u^{*},v^{*})=(0.125,0.52)$ (resp. $E^{*}=(0.1256,0.0.5196)$), which is confirmed by the numerical results depicted in Figure \ref{Ex3_Fig6} (resp. Figure \ref{Ex3_Fig8}) with various values of initial data. \\
	\end{itemize}
	\begin{figure}[tbp]
		%\centering
		\includegraphics[width=6.5in]{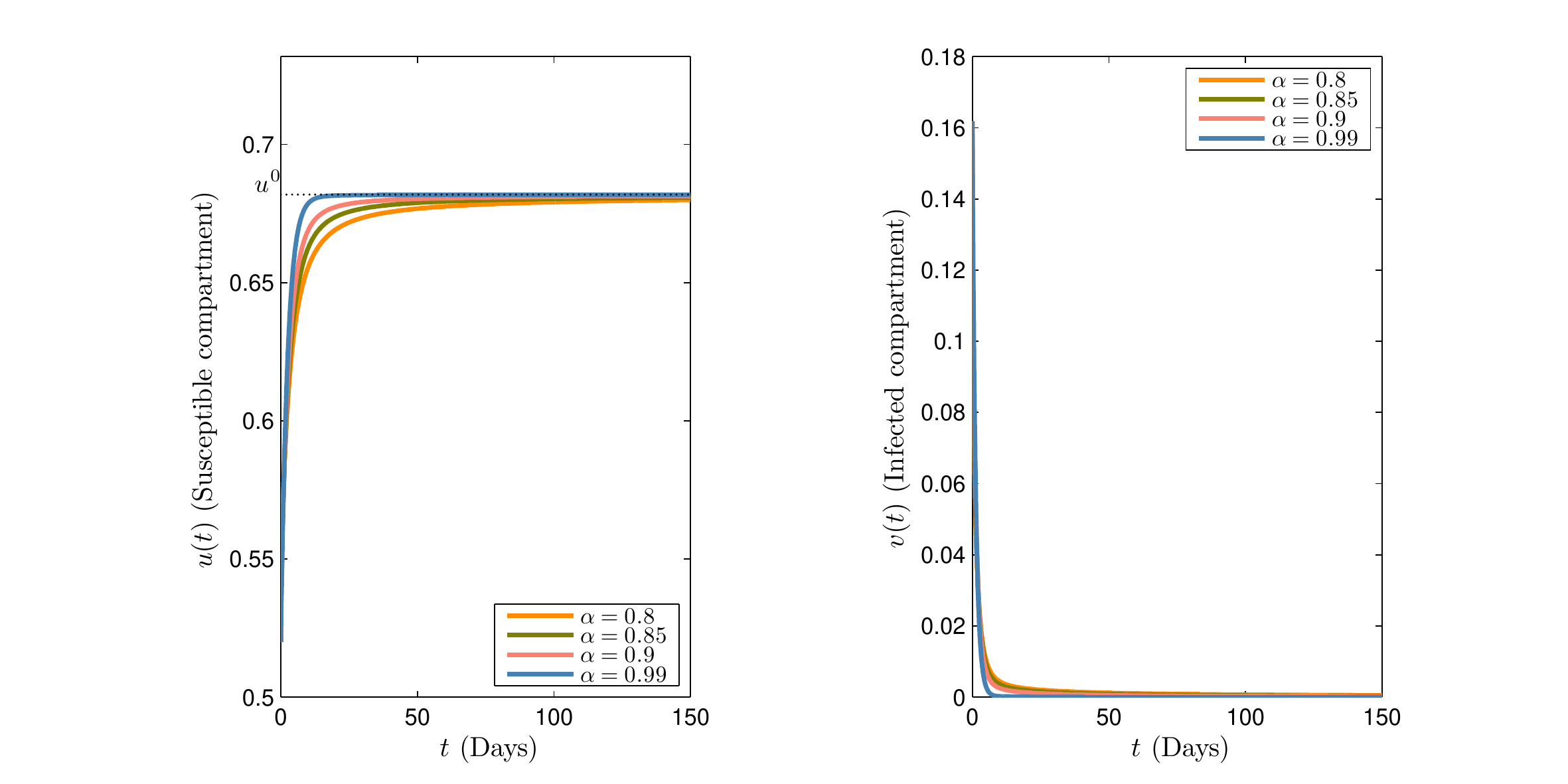}
		\caption{The approximate solution of model \eqref{eq31} by PPC (\ref{y_corrector})-(\ref{eq216}), subject to the first set of parameters in table \ref{Table_parametersSI}, with $\varphi(v)=v$ and $( u_{0},v_{0})=\left( 0.52,\frac{\Lambda}{\mu}-0.52\right) $, considering different values of fractional order $\alpha $.}
		\label{Ex3_Fig1}
	\end{figure}
	
	\begin{figure}[tbp]
		%\centering
		\includegraphics[width=6.5in]{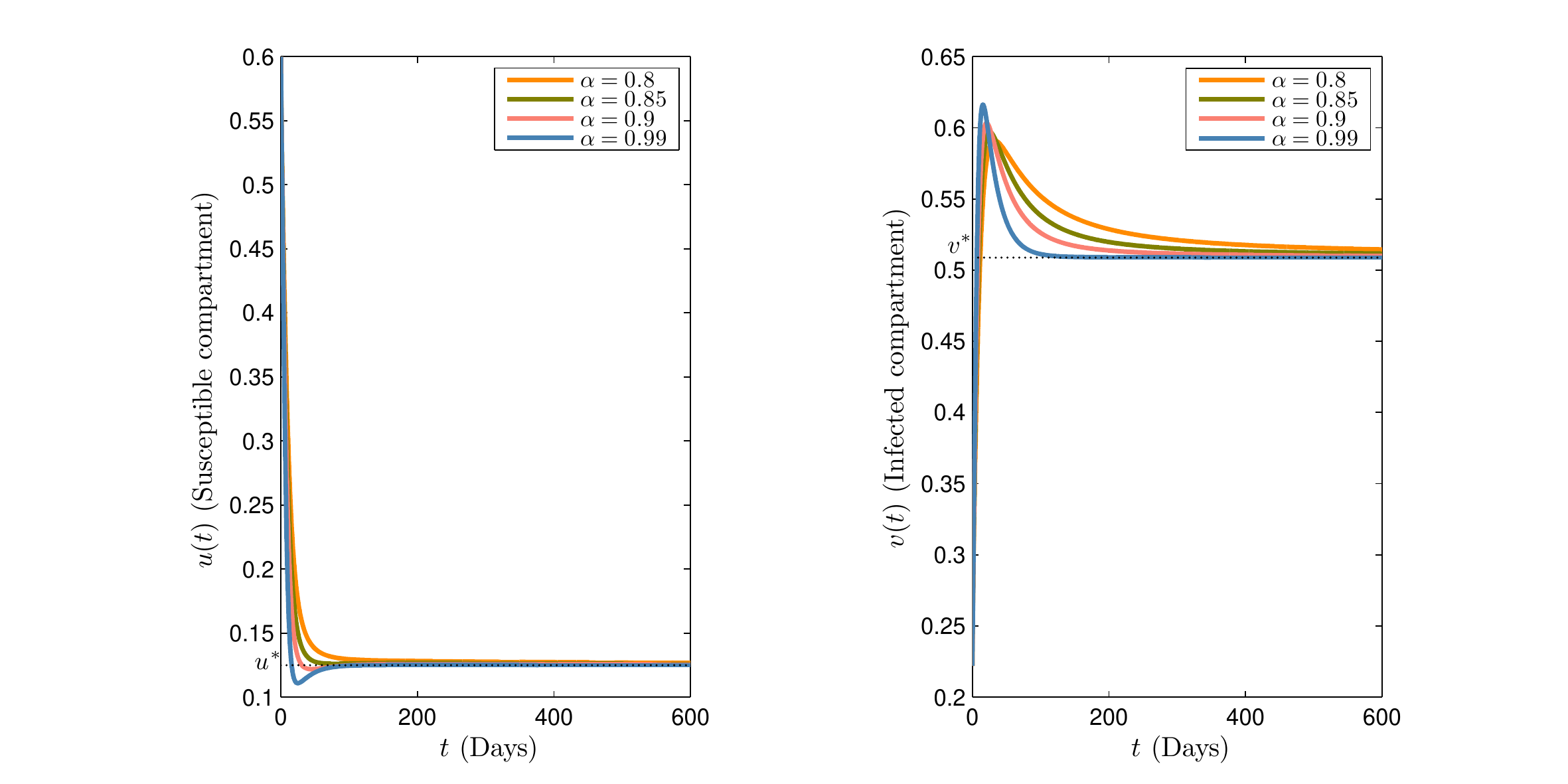}
		\caption{The approximate solution of model \eqref{eq31} by PPC (\ref{y_corrector})-(\ref{eq216}), subject to the second set of parameters in table \ref{Table_parametersSI}, with $\varphi(v)=v$ and $(u_{0},v_{0})=\left( 0.6,\frac{\Lambda}{\mu}-0.6\right) $, considering different values of fractional order $\alpha\in\left\lbrace 0.8,0.85,0.9,0.99 \right\rbrace $.}
		\label{Ex3_Fig2}
	\end{figure}
	
	\begin{figure}[tbp]
		%\centering
		\includegraphics[width=6.5in]{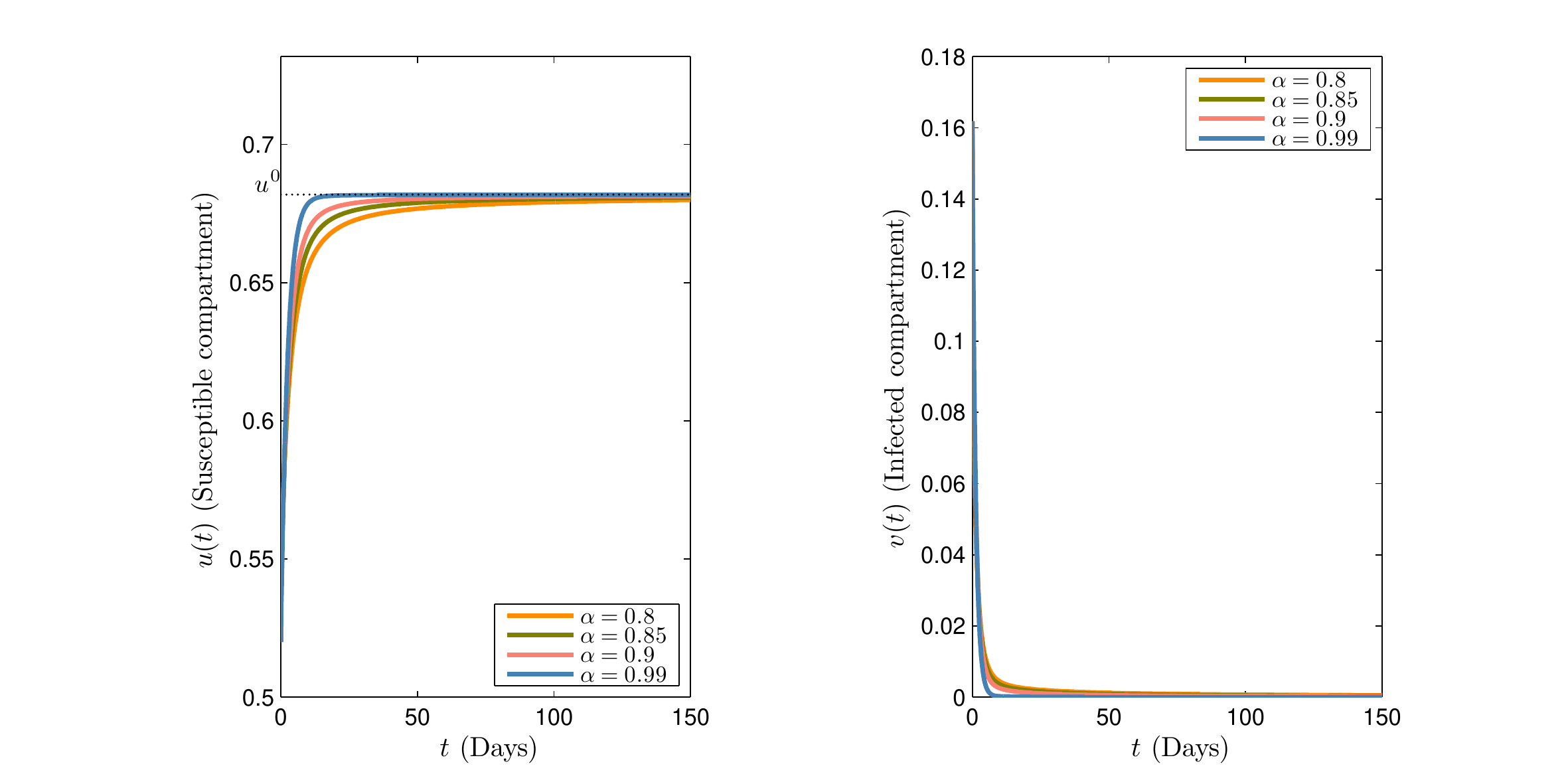}
		\caption{The approximate solution of model \eqref{eq31} by PPC (\ref{y_corrector})-(\ref{eq216}), subject to the first set of parameters in table \ref{Table_parametersSI}, with $\varphi(v)=\displaystyle\frac{v}{1+0.01v}$ and $(u_{0},v_{0})=\left( 0.52,\frac{\Lambda}{\mu}-0.52\right) $, considering different values of fractional order $\alpha\in\left\lbrace 0.8,0.85,0.9,0.99 \right\rbrace $.}
		\label{Ex3_Fig3}
	\end{figure}
	
	\begin{figure}[tbp]
		%\centering
		\includegraphics[width=6.5in]{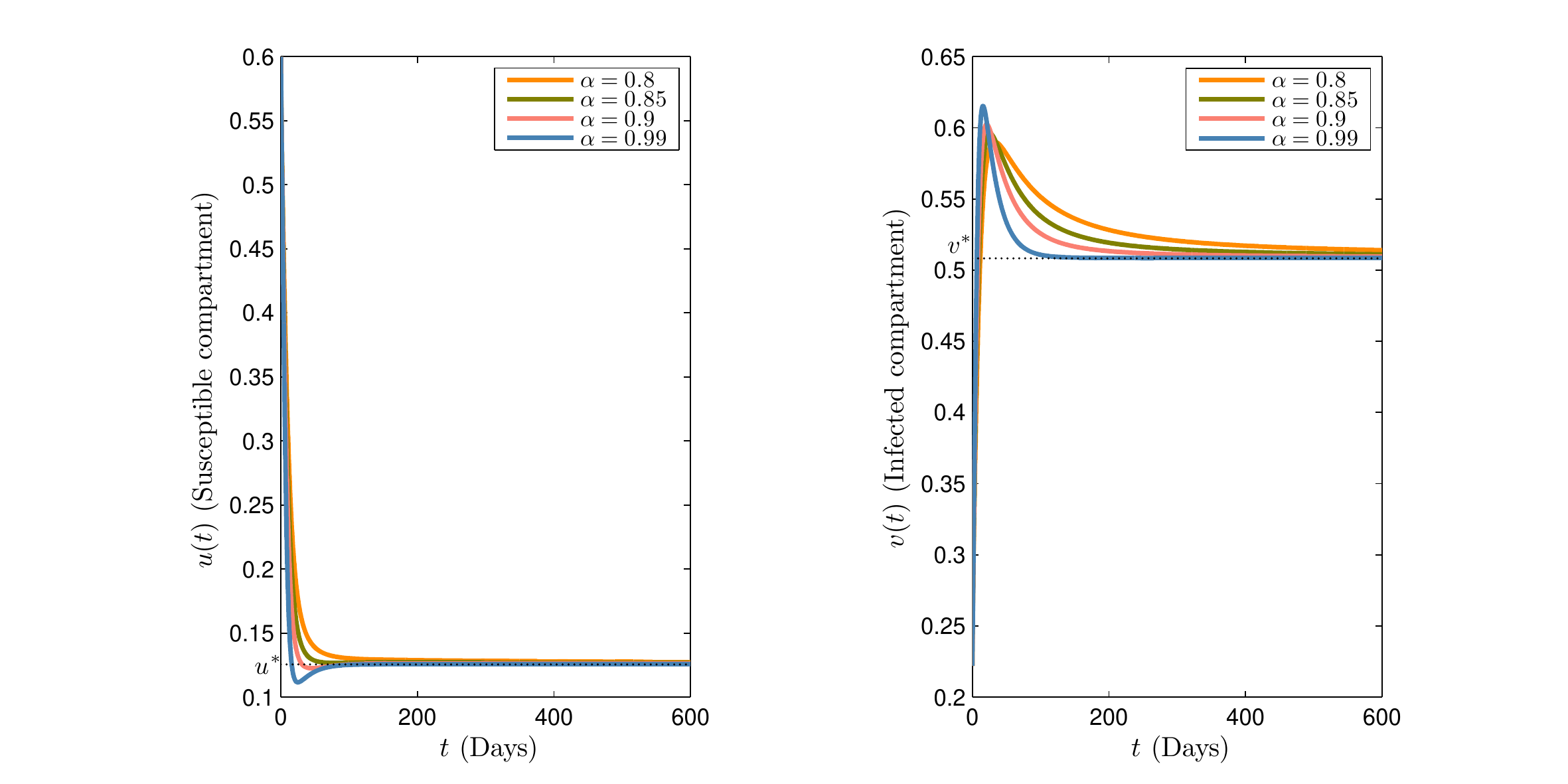}
		\caption{The approximate solution of model \eqref{eq31} by PPC (\ref{y_corrector})-(\ref{eq216}), subject to the second set of parameters in table \ref{Table_parametersSI}, with $\varphi(v)=\displaystyle\frac{v}{1+0.01v}$ and $(u_{0},v_{0})=\left( 0.6,\frac{\Lambda}{\mu}-0.6\right) $, considering different values of fractional order $\alpha\in\left\lbrace 0.8,0.85,0.9,0.99 \right\rbrace $.}
		\label{Ex3_Fig4}
	\end{figure}
	
	\begin{figure}[tbp]
		%\centering
		\includegraphics[width=6.5in]{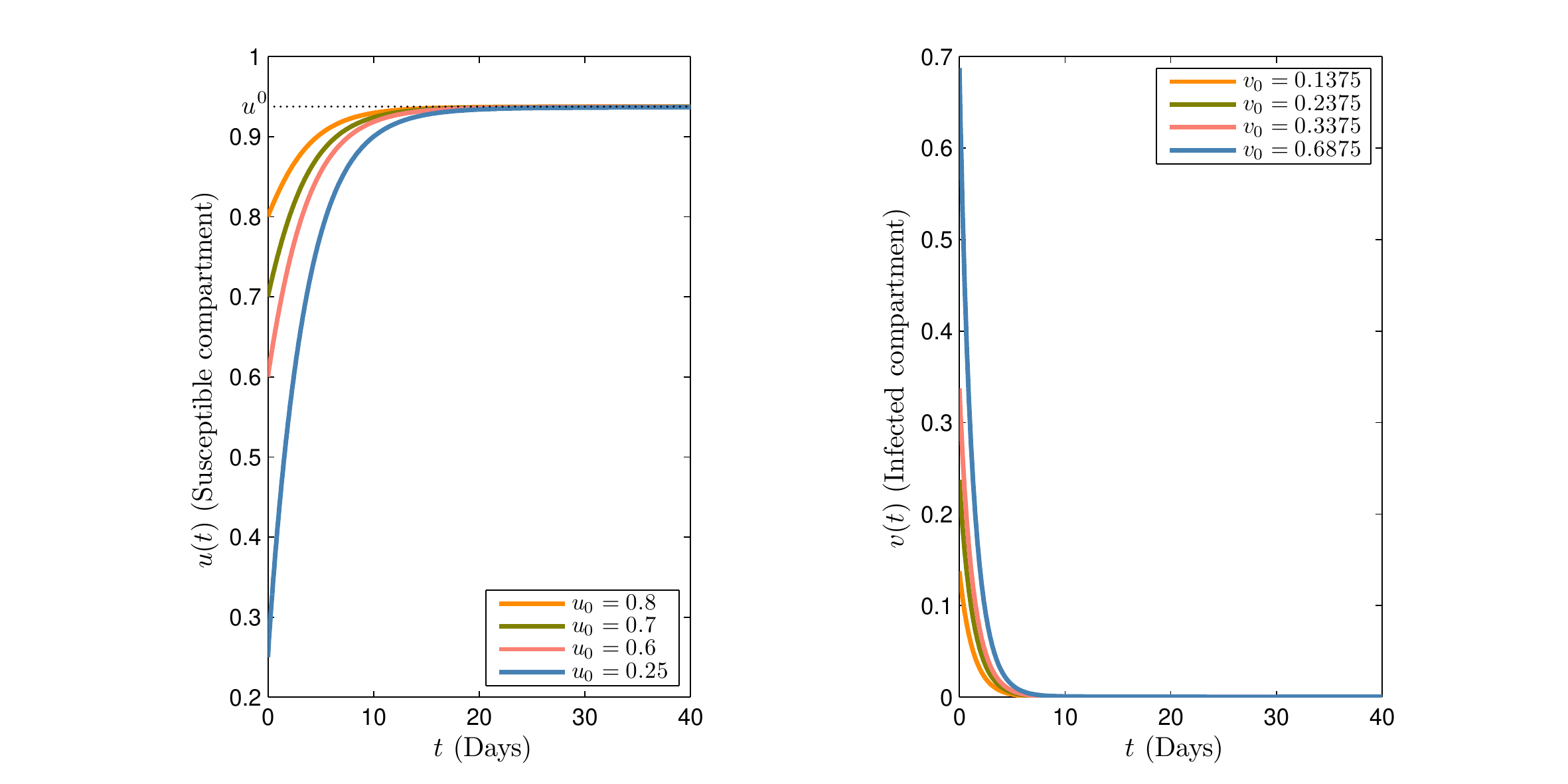}
		\caption{The approximate solution of model \eqref{eq31} by PPC (\ref{y_corrector})-(\ref{eq216}), subject to the third set of parameters in table \ref{Table_parametersSI}, with $\varphi(v)=v$ and $\alpha=0.99$, considering different values of initial data $(u_{0},v_{0})\in\left\lbrace \left( 0.8,\frac{\Lambda}{\mu}-0.8\right) ,\left( 0.7,\frac{\Lambda}{\mu}-0.7\right) ,\left( 0.6,\frac{\Lambda}{\mu}-0.6\right) ,\left( 0.25,\frac{\Lambda}{\mu}-0.25\right)  \right\rbrace $.}
		\label{Ex3_Fig5}
	\end{figure}
	
	\begin{figure}[tbp]
		%\centering
		\includegraphics[width=6.5in]{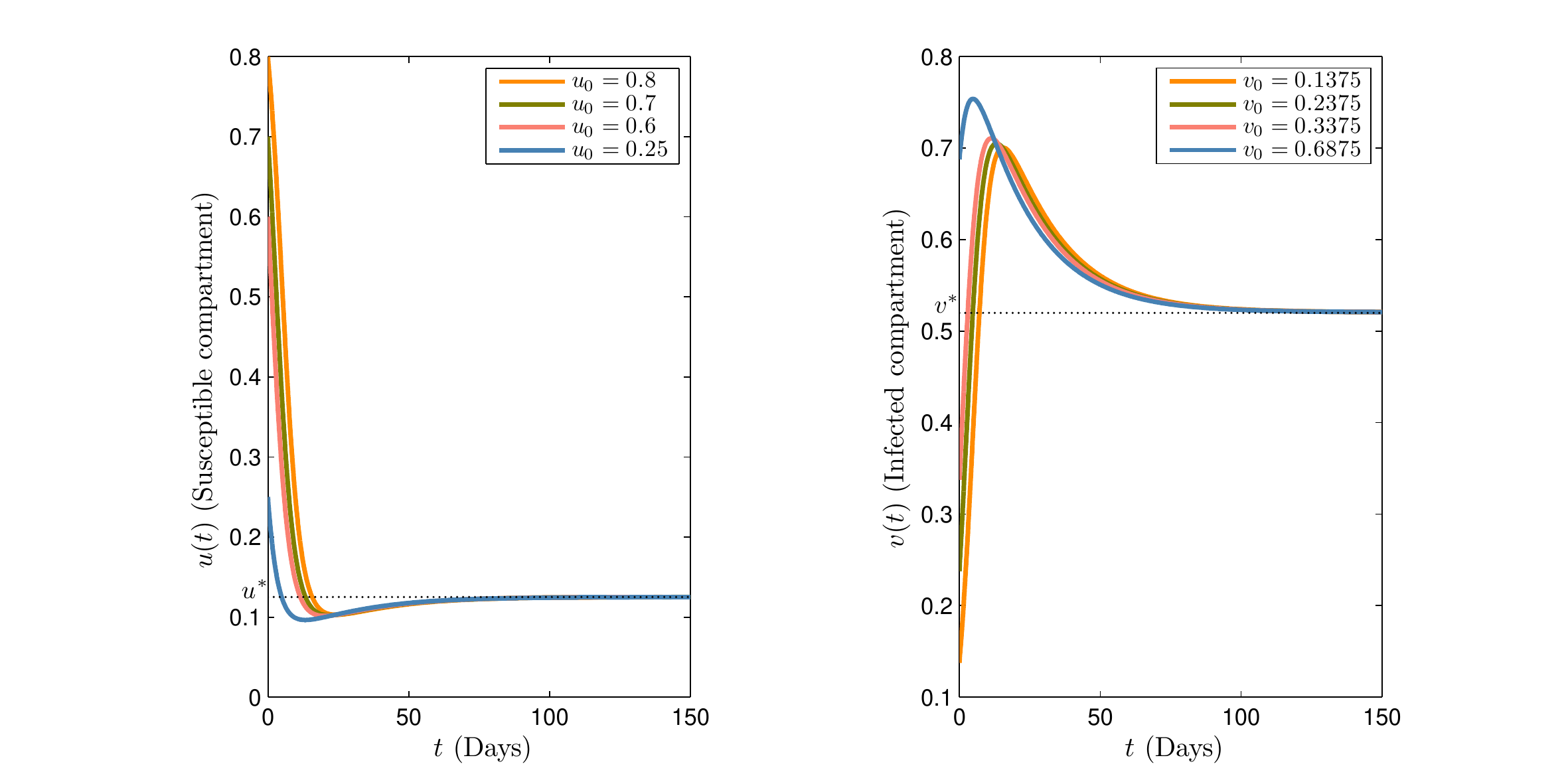}
		\caption{The approximate solution of model \eqref{eq31} by PPC (\ref{y_corrector})-(\ref{eq216}), subject to the fourth set of parameters in table \ref{Table_parametersSI}, with $\varphi(v)=v$ and $\alpha=0.99 $, considering different values of initial data $(u_{0},v_{0})\in\left\lbrace \left( 0.8,\frac{\Lambda}{\mu}-0.8\right) ,\left( 0.7,\frac{\Lambda}{\mu}-0.7\right) ,\left( 0.6,\frac{\Lambda}{\mu}-0.6\right) ,\left( 0.25,\frac{\Lambda}{\mu}-0.25\right)  \right\rbrace $.}
		\label{Ex3_Fig6}
	\end{figure}
	
	\begin{figure}[tbp]
		%\centering
		\includegraphics[width=6.5in]{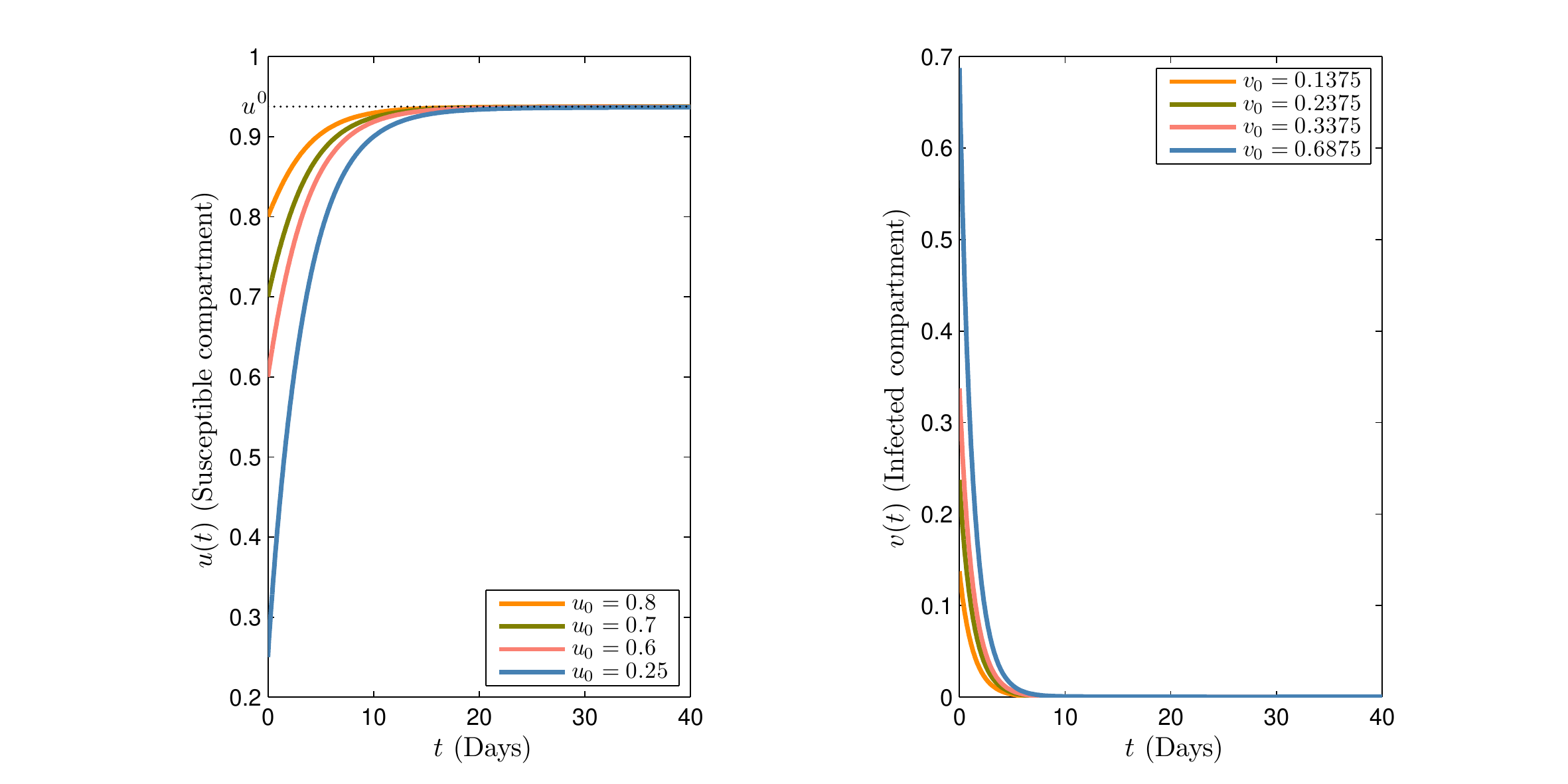}
		\caption{The approximate solution of model \eqref{eq31} by PPC (\ref{y_corrector})-(\ref{eq216}), subject to the third set of parameters in table \ref{Table_parametersSI}, with $\varphi(v)=\displaystyle\frac{v}{1+0.01v}$ and $\alpha=0.99$, considering different values of initial data $(u_{0},v_{0})\in\left\lbrace \left( 0.8,\frac{\Lambda}{\mu}-0.8\right) ,\left( 0.7,\frac{\Lambda}{\mu}-0.7\right) ,\left( 0.6,\frac{\Lambda}{\mu}-0.6\right) ,\left( 0.25,\frac{\Lambda}{\mu}-0.25\right)  \right\rbrace $.}
		\label{Ex3_Fig7}
	\end{figure}
	
	\begin{figure}[tbp]
		%\centering
		\includegraphics[width=6.5in]{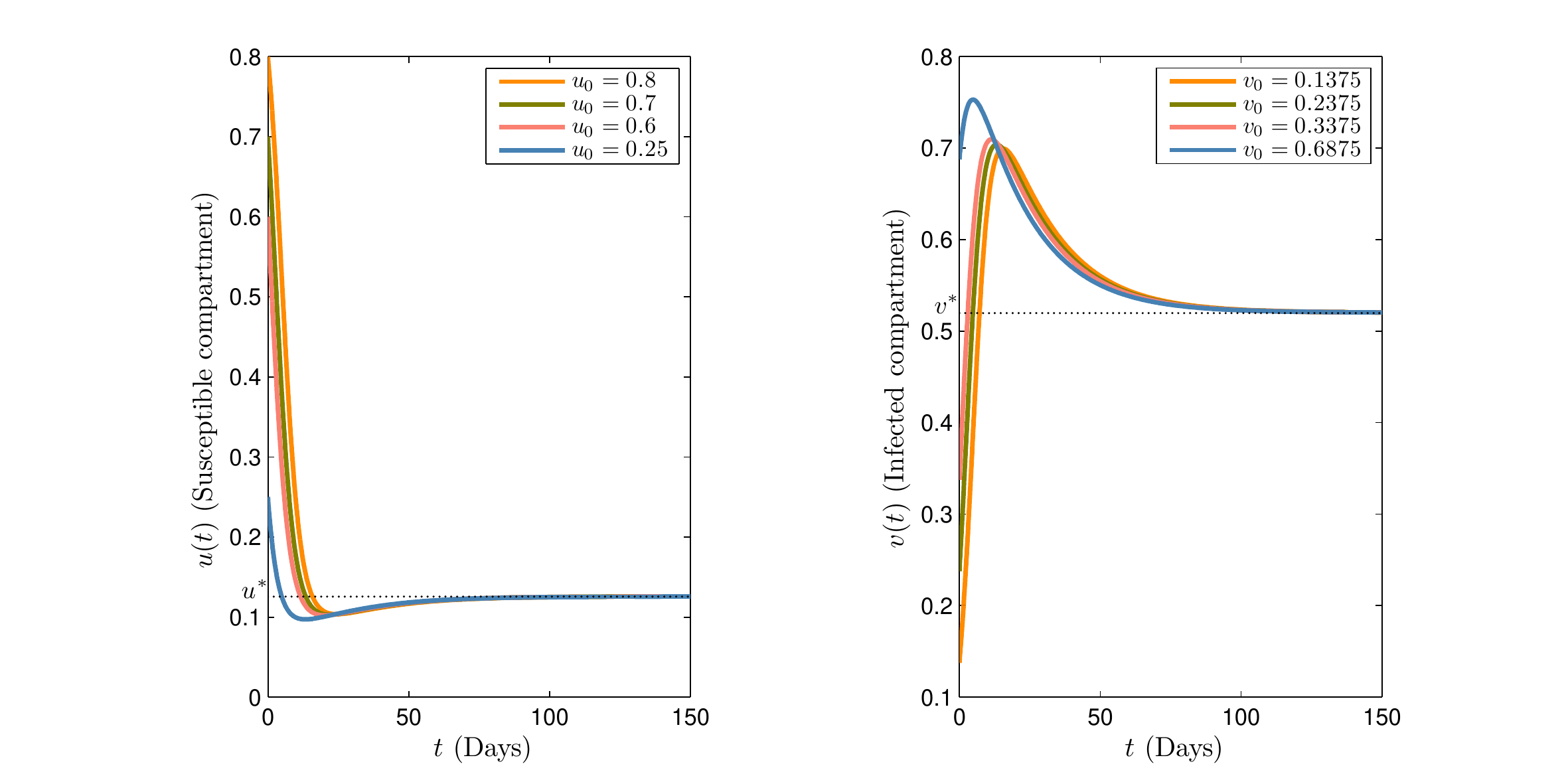}
		\caption{The approximate solution of model \eqref{eq31} by PPC (\ref{y_corrector})-(\ref{eq216}), subject to the fourth set of parameters in table \ref{Table_parametersSI}, with $\varphi(v)=\displaystyle\frac{v}{1+0.01v}$ and $\alpha=0.99 $, considering different values of initial data $(u_{0},v_{0})\in\left\lbrace \left( 0.8,\frac{\Lambda}{\mu}-0.8\right) ,\left( 0.7,\frac{\Lambda}{\mu}-0.7\right) ,\left( 0.6,\frac{\Lambda}{\mu}-0.6\right) ,\left( 0.25,\frac{\Lambda}{\mu}-0.25\right)  \right\rbrace $.}
		\label{Ex3_Fig8}
	\end{figure}
\end{example}

\section*{Acknowledgement}
The authors are grateful to Pr. Bongsoo Jang and Mr. Junseo Lee for sharing the algorithms file related to the methods developed in \cite{Nguyen2017}. The work of S. Aljhani is supported by Taibah University - Saudi Arabia.

\newpage
\section*{Appendix A. Start-up of the Scheme}
To find a desired accuracy for $\widetilde{y}_{1}$ and $\widetilde{y}_{2}$, we find the approximate solutions at points $t_{\frac{1}{4}}$ and $t_{\frac{1}{2}}$ using the constant, linear and quadratic interpolation. Let $I^{0}(f_{a})$ be the constant interpolation of $f$ at point $t=a$ that is  $I^{0}(f_{a}) = f(a)$. Let $I^{1} (f_{a},f_{b})$ and $I^{2}(f_{a},f_{b},f_{c})$ be linear and quadratic interpolation of $f$ with grids $(a,b)$ and $(a,b,c)$ respectively. In the algorithm below we describe how to find the approximate solution of $y(t)$ at points $t_{\frac{1}{4}},t_{\frac{1}{2}},t_{1},$ and $t_{2}$ (i.e. $\widetilde{y}_{\frac{1}{4}},\widetilde{y}_{\frac{1}{2}},\widetilde{y}_{1},$ and $\widetilde{y}_{2}$) using a predictor-corrector scheme: 
\begin{itemize}
	\item 
	Approximate solution of $y(t)$ at point $t_{\frac{1}{4}}$:
	\begin{align*}
	\widetilde{y}^{P}_{\frac{1}{4}} &= y_{0} + AB_{f} \widetilde{f}(t_{\frac{1}{4}},\widetilde{y}_{0}) + AB_{y} \int_{0}^{\frac{1}{4}} (t_{\frac{1}{4}} - s)^{\alpha-1} I^{0}(\widetilde{f}_{0}) ds,   \\
	\widetilde{y}_{\frac{1}{4}} & = y_{0} + AB_{f} \widetilde{f}^{P}_{\frac{1}{4}} + AB_{g} \int_{0}^{t_{\frac{1}{4}}} I^{1}(\widetilde{f}_{0}, \widetilde{f}^{P}_{\frac{1}{4}})(t_{\frac{1}{4}} - s)^{\alpha-1} ds  . 
	\end{align*}
	
	\item 
	Approximate solution of $y(t)$ at point $t_{\frac{1}{2}}$:
	\begin{align*}
	\widetilde{y}^{P_{1}}_{\frac{1}{2}} &= y_{0} + AB_{f}(2\widetilde{f}_{\frac{1}{4}} - \widetilde{f}_{0}) + AB_{g} \int_{0}^{t_{\frac{1}{2}}} I^{0}(\widetilde{f}_{\frac{1}{4}})(t_{\frac{1}{2}} - s)^{\alpha-1} ds,  \\ 
	\widetilde{y}^{P_{2}}_{\frac{1}{2}} &= y_{0} + AB_{f} \widetilde{f}^{P_{1}}_{\frac{1}{2}} + AB_{g} \int_{0}^{t_{\frac{1}{2}}} I^{1}(\widetilde{f}_{\frac{1}{4}} , \widetilde{f}^{P_{1}}_{\frac{1}{2}})(t_{\frac{1}{2}}-s)^{\alpha-1} ds, \\ 
	\widetilde{y}_{\frac{1}{2}} &= y_{0} + AB_{f} \widetilde{f}^{P_{2}}_{\frac{1}{2}} + AB_{g} \int_{0}^{t_{\frac{1}{2}}} I^{2}(\widetilde{f}_{0}, \widetilde{f}_{\frac{1}{4}},\widetilde{f}^{P_{2}}_{\frac{1}{2}})(t_{\frac{1}{2}} - s)^{\alpha-1} ds.
	\end{align*}
	
	\item 
	Approximate solution of $y(t)$ at point $t_{1}$:
	\begin{align*}
	\widetilde{y}^{P_{1}}_{1} &= y_{0} + AB_{f}(6\widetilde{f}_{\frac{1}{2}} - 8\widetilde{f}_{\frac{1}{4}}-3\widetilde{f}_{0}) + AB_{g} \int_{0}^{t_{n}} I^{0}(\widetilde{f}_{\frac{1}{2}})(t_{n} - s)^{\alpha - 1} ds, \\ 
	\widetilde{y}^{P_{2}}_{1} &= y_{0} + AB_{f} \widetilde{f}^{P_{1}}_{1} + AB_{g} \int_{0}^{t_{1}} I^{1}(\widetilde{f}_{\frac{1}{2}}, \widetilde{f}^{P_{1}}_{1}) (t_{1} - s)^{\alpha-1} ds, \\ 
	\widetilde{y}_{1} &= y_{0} + AB_{f} \widetilde{f}^{P_{2}}_{1} + AB_{g} \int_{0}^{t_{1}}I^{2}(\widetilde{f}_{0},\widetilde{f}_{\frac{1}{2}},\widetilde{f}^{P_{2}}_{1})(t_{1}-s)^{\alpha-1} ds.
	\end{align*}
	
	\item 
	Approximate solution of $y(t)$ at point $t_{2}$:
	\begin{align*}
	\widetilde{y}^{P_{1}}_{2} &= y_{0} + AB_{f}(6\widetilde{f}_{1} - 8\widetilde{f}_{\frac{1}{2}}-3\widetilde{f}_{0}) + \int_{0}^{t_{2}} I^{0} (\widetilde{f}_{1})(t_{2}-s)^{\alpha-1} ds, \\ 
	\widetilde{y}^{P_{2}}_{2} &= y_{0} + AB_{f} \widetilde{f}^{P_{1}}_{2} + \int_{0}^{t_{2}} I^{1} (\widetilde{f}_{1},\widetilde{f}^{P_{1}}_{2})(t_{2}-s)^{\alpha-1} ds, \\ 
	\widetilde{y}_{2} &= y_{0} + AB_{f} \widetilde{f}^{P_{2}}_{2} + \int_{0}^{t_{2}} I^{2}(\widetilde{f}_{0},\widetilde{f}_{1},\widetilde{f}^{P_{2}}_{2}) (t_{2}-s)^{\alpha-1} ds.
	\end{align*}
\end{itemize}

where
$$\Biggl\{ AB_{f} = \frac{1-\alpha}{AB(\alpha)} \,\,\,\,\,\textit{and}\,\,\,\,\, AB_{g} = \frac{\alpha}{AB(\alpha)\Gamma(\alpha)} \Biggr\}.$$

\end{document}